\documentclass[12pt,reqno,NoRunningHeads]{amsart}
\usepackage{latexsym}
\usepackage{amsmath}
\usepackage{amssymb}
\usepackage{amsfonts}
\usepackage{verbatim}
\usepackage[latin1]{inputenc}
\usepackage{color}
\usepackage{graphicx}

\numberwithin{equation}{section}
\newtheorem{theorem}{Theorem}[section]
\newtheorem{lemma}[theorem]{Lemma}

\theoremstyle{definition}
\newtheorem{definition}{Definition}[section]

\theoremstyle{remark}
\newtheorem{remark}{\bf Remark}[section]

\newtheorem{corollary}[theorem]{\bf Corollary}

\newtheorem{assumption}{\bf Assumption}[section]

\let \ssection=\section
\renewcommand{\section}{\setcounter{equation}{0}\ssection}

\def\^#1{\if#1i{\accent"5E\i}\else{\accent"5E #1}\fi}
\def\"#1{\if#1i{\accent"7F\i}\else{\accent"7F #1}\fi}

\newcommand{\bN}{\mathbb N}
\newcommand{\RR}{\mathbb R}
\newcommand{\bR}{\mathbb R}

\newcommand{\ZZ}{\mathbb Z}

\newcommand{\cI}{\mathcal I}
\newcommand{\cF}{\mathcal F}
\newcommand{\cM}{\mathcal M}
\newcommand{\cL}{\mathcal L}
\newcommand{\cK}{\mathcal K}

\newcommand{\cP}{\mathcal P}

\newcommand{\bQ}{\mathbb Q}

\newcommand{\bA}{\mathbb A}
\newcommand{\bB}{\mathbb B}
\newcommand{\bG}{\mathbb G}
\newcommand{\bW}{\mathbb W}

\def\^#1{\if#1i{\accent"5E\i}\else{\accent"5E #1}\fi}
\def\"#1{\if#1i{\accent"7F\i}\else{\accent"7F #1}\fi}
\allowdisplaybreaks

\begin{document}
 \title[Accelerated finite elements schemes] 
{Accelerated finite elements schemes
for parabolic stochastic partial differential equations} 
\author[I. Gy\"ongy]
{Istv\'an Gy\"ongy}

\address{Maxwell Institute and School of Mathematics,
University of Edinburgh,
King's  Buildings,
Edinburgh, EH9 3JZ, United Kingdom}
\email{gyongy@maths.ed.ac.uk}

 \author[A. Millet]{Annie Millet}
\address
{SAMM (EA 4543), Universit\'e Paris 1 Panth\'eon Sorbonne,
 90 Rue de Tolbiac, 75634 Paris Cedex 13, France {\it  and }
Laboratoire de Probabilit\'es, Statistique et Mod\'elisation 
(LPSM, UMR 8001).}
\email { annie.millet@univ-paris1.fr {\it and} millet@lpsm.paris}



\begin{abstract}
For a class of finite elements approximations for linear stochastic parabolic PDEs 
it is proved that one can accelerate the rate of convergence by Richardson extrapolation. 
More precisely, by taking appropriate mixtures of finite elements approximations 
one can accelerate the convergence to any given speed  provided the coefficients,  
 the initial and free data are sufficiently smooth.  
\end{abstract}

\maketitle

\smallskip

\section{Introduction}                                                                     \label{intro}
We are interested in finite elements approximations 
for Cauchy problems for stochastic parabolic PDEs of the form 
of equation \eqref{SPDE1} below. Such kind of equations arise 
in various fields of sciences and engineering, for example in 
nonlinear filtering of partially observed diffusion processes. 
Therefore these equations have been intensively studied in 
the literature, and theories for their solvability and 
numerical methods for approximations of their solutions 
have been developed. Since the computational effort to get 
reasonably accurate numerical solutions grow rapidly with the 
dimension $d$ of the state space, it is important to investigate the possibility 
of accelerating the convergence of spatial discretisations 
by Richardson extrapolation. About a century ago Lewis Fry Richardson 
had the idea in \cite{R1911} that the speed of convergence of 
numerical approximations, which 
depend on some parameter $h$ converging to zero, can be increased
if one takes appropriate linear combinations of approximations corresponding to different parameters.
 This method to accelerate the convergence, called 
Richardson extrapolation, works when the approximations admit a power series  expansion 
in $h$ at $h=0$ with a remainder term, which can be estimated by a higher power of $h$.  
In such cases, taking appropriate mixtures of approximations with different parameters, one can eliminate all 
other terms but the 
zero order term and the remainder in the expansion. In this way,   the order of 
accuracy of the mixtures is the exponent $k+1$ of the power $h^{k+1}$, that  
estimates the remainder. 
For various numerical methods applied to solving deterministic partial differential equations (PDEs)
it has been proved that such expansions exist 
and that Richardson extrapolations 
can spectacularly increase the speed of convergence of the methods, 
see, e.g., \cite{Ran1}, \cite{Ran2} and  \cite{S}. Richardson's idea has also been applied to numerical 
solutions of stochastic equations. It was shown first in \cite{TT} that by Richardson 
extrapolation one can accelerate the weak convergence of Euler approximations 
of stochastic differential equations. Further results in this direction can be found in 
\cite{LP}, \cite{MT} and the references therein. For stochastic PDEs the first result on accelerated finite 
difference schemes appears in \cite{GK}, where it is shown that by Richardson 
extrapolation one can accelerate the speed of finite difference schemes in the 
spatial variables for linear stochastic parabolic PDEs to any high order, provided 
the initial condition and free terms are sufficiently smooth. This result was extended 
to (possibly) degenerate stochastic PDEs in  
to \cite{GG}, \cite{H1} and \cite{H2}. Starting with \cite{W2005} finite 
elements approximations for stochastic PDEs have been investigated in many 
publications, see, for example, 
\cite{BrzPro}, \cite{CaPr}, \cite{Haus},  
\cite{KLL}, \cite{KP} and \cite{Y}. 

Our main result, Theorem  \ref{theorem main2} in this paper,  states 
that for a class of finite elements approximations 
for stochastic parabolic PDEs given in the whole space 
an expansion in terms of powers of a parameter $h$, 
proportional to the size of the finite elements, exists up to any high order, 
if the coefficients, the initial data and the free terms are sufficiently smooth.
 Then clearly, we can apply Richardson extrapolation 
 to such finite elements approximations in order to 
 accelerate the convergence. The speed we can achieve depends on 
 the degree of smoothness of the coefficients, the initial data and free terms; 
 see Corollary \ref{corollary1}.
 Note that due to the symmetry we require for the finite elements, 
 in order to achieve an accuracy of order $J+1$
 we only need $\lfloor \frac{J}{2} \rfloor$  
 terms in the mixture of finite elements approximation. 
As far as we know this is the first result on accelerated finite elements by Richardson 
extrapolation for stochastic parabolic equations. 
There are nice results on Richardson extrapolation for finite elements 
schemes in the literature for some (deterministic) elliptic problems; 
see,  e.g., \cite{ASW}, \cite{BLR}  and the literature therein. 

We note that in the present paper we consider stochastic PDEs on the whole 
space $\bR^d$ in the spatial variable, and our finite elements approximations 
are the solutions of infinite dimensional systems of equations. 
Therefore one may think that our accelerated finite elements 
schemes cannot have any practical use. 
In fact they can be implemented if first we localise 
the stochastic PDEs in the spatial variable by multiplying their coefficients, 
initial and free data by sufficiently smooth non-negative 
``cut-off" functions with value 1 on a ball of radius $R$ 
and vanishing outside of a bigger ball. 
Then our finite elements schemes 
for the ``localised stochastic PDEs" are fully implementable and one can show 
that the results of the present paper can be extended to them. Moreover, 
by a theorem from \cite{GG} the error caused by the localization 
is of order $\exp(-\delta R^2)$  within a ball of radius $R'<R$. 
Moreover, under some further constraints about a bounded domain 
$D$ and particular classes
of finite elements such as those described in subsections  
\ref{fel1}-\ref{general_example},
our arguments could extend to parabolic stochastic PDEs on $D$ 
with periodic boundary conditions. 
Note that our technique relies on finite elements defined 
by scaling and shifting  one given mother element,
and that the dyadic rescaling used to achieve a given speed of 
convergence is similar to that of wavelet approximation. 
We remark that our accelerated finite elements approximations  
can be applied also to implicit Euler-Maruyama time discretisations 
of stochastic parabolic PDEs to achieve higher order convergence 
with respect to the spatial mesh parameter of fully discretised schemes. 
However, as one can see by adapting and argument from \cite{DG}, 
the strong rate of convergence of these fully discretised schemes 
with respect to the temporal mesh parameter  cannot be accelerated 
by Richardson approximation. Dealing with weak speed of convergence 
of time discretisations is beyond the scope of this paper.

In conclusion we introduce some notation used in the paper. All random elements 
are defined on a fixed probability space $(\Omega,\cF,P)$ equipped with 
an increasing family $(\cF_t)_{t\geq0}$ of $\sigma$-algebras 
$\cF_t\subset\cF$. The predictable $\sigma$-algebra of subsets of 
$\Omega\times[0,\infty)$ is denoted by $\cP$, and the $\sigma$-algebra 
of the Borel subsets of $\bR^d$ is denoted by ${\mathcal B}(\bR^d)$. We use the notation
$$
D_i=\frac{\partial}{\partial x_i}, 
\quad 
D_{ij}=D_iD_j=\frac{\partial^2}{\partial x_i\partial x_j}, 
\quad
i,j=1,2,...,d
$$
for first order and second order partial derivatives in  $x=(x_1,...,x_d)\in\bR^d$. 
For integers $m\geq0$  the Sobolev space $H^m$ is defined as the closure 
of $C_0^{\infty}$, the space of real-valued smooth functions $\varphi$ on $\bR^d$ with 
compact support, in the norm $|\varphi|_m$ defined by 
\begin{equation}                                                \label{H}
|\varphi|_m^2=\sum_{|\alpha|\leq m}\int_{\bR^d}|D^{\alpha}\varphi(x)|^2\,dx, 
\end{equation}
where $D^{\alpha}=D_1^{\alpha_1}... D_d^{\alpha_d}$ and 
$|\alpha|=\alpha_1+\cdots +\alpha_d$ for multi-indices 
$\alpha=(\alpha_1,...,\alpha_d)$, $\alpha_i\in\{0,1,...,d\}$, 
and $D_i^0$ is the identity operator for $i=1,...,d$. 
Similarly, the Sobolev space $H^m(l_2)$ of $l_2$-valued functions are defined on $\bR^d$ 
as the closure of the of $l_2$-valued smooth functions 
$\varphi=(\varphi_i)_{i=1}^{\infty}$ on $\bR^d$ with compact support, in the norm 
denoted also by $|\varphi|_m$ and defined as in \eqref{H} with  
$
\sum_{i=1}^{\infty}| D^{\alpha}\varphi_i(x)|^2
$
in place of $|D^{\alpha}\varphi(x)|^2$. 
Unless stated otherwise, throughout the paper we use the summation convention with respect to repeated 
indices. 
The summation over an empty set means 0. 
We denote by $C$ and $N$ constants which may change from
one line to the next, and by $C(a)$ and $N(a)$ constants depending 
on a parameter $a$.

For theorems and notations 
in the $L_2$-theory of stochastic PDEs the reader is referred to \cite{KrRo} 
or \cite{R}. 

\section{Framework and some notations}                                            \label{framework} 
Let $(\Omega,\mathcal F,P,({\mathcal F}_t)_{t\geq0})$ be a 
complete  filtered probability space carrying  
 a sequence of independent 
Wiener martingales $W=(W^{\rho})_{\rho=1}^{\infty}$ 
with respect to a filtration $({\mathcal F}_t)_{t\geq0}$. 

We consider the stochastic PDE problem 
\begin{equation}                                                                     \label{SPDE1}                                                          
 d u_t(x) =\big[ {\mathcal L}_t u_t(x) + f_t(x)\big] dt 
 + \big[ {\mathcal M}_t^{\rho}u_t(x) + g^{\rho}_t(x) \big] dW_t^{\rho} ,\quad 
 (t,x)\in[0,T]\times\mathbb R^d, 
 \end{equation}
 with initial condition
 \begin{equation}                                                                       \label{ini}                               
 u_0(x)=\phi(x),\quad       x\in\mathbb R^d,                                                   
\end{equation}
for a given $\phi\in H^0=L_2(\mathbb R^d)$, where  
\begin{eqnarray*}
{\mathcal L}_t u (x) &=&D_{i}( a^{ij}_t(x)D_{j} u(x))+b^i_t(x)D_i u(x)+c_t(x)u(x) , \\
{\mathcal M}^{\rho}_t u(x) &=& \sigma^{i\rho}_t(x)D_{i} u(x) +\nu^{\rho}_t(x)u(x) \quad 
\text{for } u\in H^1=W^1_2(\mathbb R^d), 
\end{eqnarray*}
with $\mathcal P\otimes{\mathcal B}(\mathbb R^d)$-measurable 
real-valued bounded functions $a^{ij}$, $b^i$, $c$, and $l_2$-valued 
bounded functions $\sigma^{i}=(\sigma^{i\rho})_{\rho=1}^{\infty}$ and 
$\nu=(\nu^{\rho})_{\rho=1}^{\infty}$
defined on $\Omega\times[0,T]\times\mathbb R^d$ for $i,j\in\{1, ..., d\}$. 
Furthermore, $a^{ij}_t(x)=a^{j i}_t(x)$ a.s. for every $(t,x)\in [0,T]\times \RR^d$. 
For $i=1,2, ..., d$ the notation $D_i = \frac{\partial}{\partial x_i}$ means 
the partial derivative in the $i$-th coordinate direction.   

The free terms $f$ and $g=(g^{\rho})_{\rho=1}^{\infty}$ 
are $\mathcal P\otimes{\mathcal B}(\mathbb R^d)$-measurable 
functions on $\Omega\times[0,T]\times\mathbb R^d$, 
with values in $\mathbb R$ and $l_2$ respectively. 
Let  $H^{m}(l_2)$ denote the $H^m$ space of $l_2$-valued 
functions on $\mathbb R^d$. 
We use the notation $|\varphi|_m$ for the $H^m$-norm of $\varphi\in H^m$ and 
of $\varphi\in H^m(l_2)$, 
and  $|\varphi|_0$  denotes the $L_2$-norm of $\varphi\in H^0=L_2$.

Let $m\geq0$ be an integer, $K\geq0$ be a constant 
and make the following assumptions.

\begin{assumption}                                                                        \label{assumption coeff}
{\it The derivatives in $x\in\bR^d$ up to order $m$ 
of the coefficients $a^{ij}$, $b^i$, $c$, and of the coefficients 
$\sigma^{i}$, $\nu$ are 
$\mathcal P\otimes{\mathcal B}(\mathbb R)$-measurable 
functions  with values in $\mathbb R$ and in $l_2$-respectively.
For almost every $\omega$ they are continuous   in $x$,
and they are bounded in magnitude by $K$. }
\end{assumption}

\begin{assumption}                                                                 \label{assumption data}
{\it The function $\phi$ is an 
$H^m$-valued ${\mathcal F}_0$-measurable  random variable, and 
 $f$ and $g=(g^{\rho})_{\rho=1}^{\infty}$ 
are predictable processes with values in $H^{m}$ and $H^m(l_2)$, 
respectively, such that 
 \begin{equation} 		\label{Km}
{\mathcal K}^2_m:=|\phi|^2_m+
\int_0^T \big( |f_t|_{m}^2+|g_t|^2_{m} \big) \,dt  < \infty \,\,(a.s.).  
\end{equation}   }   
\end{assumption}

\begin{assumption}                                                           \label{assumption parab}
{\it There exists a constant $\kappa>0$, 
such that for $(\omega,t,x)\in \Omega \times [0,T] \times \RR^d$ 
\begin{equation}                                                          \label{2.19.5.12}
\sum_{i,j=1}^d \big( a^{ij}_t(x)
 -\tfrac{1}{2}\sum_\rho \sigma^{i\rho}_t(x)\sigma^{j\rho}_t(x) \big) z^i z^j  
 \geq\kappa |z|^2 
 \quad\text{for all $z=(z^1,...,z^d)\in\mathbb R$}.
\end{equation}
}
\end{assumption}  

For integers $n\geq0$ let $\bW^n_2(0,T)$ denote the space of 
$H^n$-valued predictable processes $(u_t)_{t\in[0,T]}$ such that 
almost surely 
$$
\int_0^T|u_t|^2_{n}\,dt<\infty. 
$$
\begin{definition} 
A continuous $L_2$-valued adapted process $(u_t)_{t\in[0,T]}$ is a generalised solution 
to \eqref{SPDE1}-\eqref{ini} if it is in $\bW^1_2(0,T)$, and almost surely 
\begin{align*}
(u_t,\varphi)=& (\phi,\varphi)
+\int_0^t \big( a^{ij}_sD_{j}u_s,D^{\ast}_{i}\varphi)+(b^i_sD_iu_s+c_su_s+f_s,\varphi \big)\,ds \\
& +\int_0^t \big( \sigma_s^{i\rho} D^i u_s + \nu^\rho_s u_s +  g^{\rho}_s,\varphi \big)\,dW^{\rho}_s 
\end{align*}
for all $t\in[0,T]$ and $\varphi\in C_0^{\infty}$,  where $D^{\ast}_{i}:=-D_{i}$  
 for $i\in\{1,2,...,d\}$, and $(,)$ denotes the inner product in $L_2$.  
 \end{definition}

 For $m\geq 0$ set 
 \begin{equation}		\label{gothic_Km}
  \mathfrak{K}_m = |\phi|_m^2 + \int_0^T   \big( |f_t|_{m-1}^2 + |g_t|_m^2\big) dt.
 \end{equation}
Then the following theorem is well-known (see, e.g., \cite{R}).
\begin{theorem}                                                               \label{theorem SPDE}                 
Let Assumptions \ref{assumption coeff}, 
\ref{assumption data} and \ref{assumption parab} hold. 
Then \eqref{SPDE1}-\eqref{ini} has a unique 
generalised solution $u=(u_t)_{t\in[0,T]}$. Moreover, $u\in \bW^{m+1}_2(0,T)$, 
it is an $H^{m}$-valued continuous process, and 
$$
E\sup_{t\in[0,T]}|u_t|^2_{m}+E\int_0^T|u_t|_{m+1}^2\,dt
\leq C   E \mathfrak{K}_m,  
$$  
where $C$ is a constant depending only on $\kappa$, $d$, $T$, $m$ and $K$.  
\end{theorem}

The finite elements we consider in this paper 
are determined by a continuous real function $\psi\in H^1$ with compact 
support, and a finite set $\Lambda\subset \bQ^d$, containing the zero vector, 
such that $\psi$ and $\Lambda$ are symmetric, 
i.e., 
\begin{equation} \label{sym_psi_Lambda}
\psi(-x)=\psi(x)\; \mbox{\rm for all } \; x\in\bR^d, \quad \mbox{\rm  and} \; 
\Lambda=-\Lambda.
\end{equation} 
We assume that $|\psi|_{L_1}=1$, which can be achieved by scaling. 
For each $h\neq 0$ and $x\in\bR^d$ we set $\psi^h_x(\cdot):=\psi((\cdot-x)/h)$, 
and our set of finite 
elements is  the collection of functions $\{\psi^{h}_x:x\in\bG_h\}$, where 
$$
\bG_h:=\left\{h\sum_{i=1}^nn_i\lambda_i:\lambda_i\in\Lambda, \,n_i,n\in\bN\right\}.  
$$
Let $V_h$ denote the vector space  
$$
V_h
:=\left\{\sum_{x\in\bG_h}U(x) \psi^{h}_{x}: (U(x))_{x\in{\bG_h}}\in \ell_2(\bG_h)\right\}, 
$$
where $\ell_2(\bG_h)$ is the space of functions $U$ on $\bG_h$ such that 
\begin{equation}			\label{U_0h}
|U|_{0,h}^2 :=|h|^d  \sum_{x\in\bG_h}  U^2(x)  <\infty.
\end{equation}

\begin{definition}
 An $L_2(\mathbb R^d)$-valued 
 continuous process $u^h=(u^h_t)_{t\in[0,T]}$ 
 is a finite elements approximation of $u$ 
 if it takes values in $V_h$ 
 and almost surely 
 \begin{align}                                                          \label{finelem}
 (u^h_t,\psi^{h}_{x})=&(\phi,\psi^{h}_{x})
 +\int_0^t
 \big[ (a^{ij}_sD_{j}u^{h}_s,D^{\ast}_{i}\psi^{h}_{x})
 +(b^i_sD_iu^h_s+c_su^h_s+f_s,\psi^{h}_x)\big] \,ds                                                                    \nonumber\\
 &
 +\int_0^t
 (\sigma^{i\rho}_sD_{i}u^h_s+\nu^{\rho}_su^h_s+g^{\rho}_s,\psi^{h}_x)\,dW^{\rho}_s, 
 \end{align}
 for all $t\in[0,T]$ and $\psi_x^h$ is as above for
 $x\in \bG_h$.  
 The process $u^h$ is also called a {\it $V_h$-solution} 
  to \eqref{finelem} on $[0,T]$. 
\end{definition}

Since by definition a $V_h$-valued solution $(u^h_t)_{t\in[0,T]}$ to \eqref{finelem} 
is of the form 
$$
u^h_t(x)=\sum_{y\in\bG_h}U^h_t(y)\psi^h_y(x), \quad x\in\bR^d, 
$$
we need to solve \eqref{finelem} 
for the random field $\{U^h_t(y):y\in\bG_{h}, t\in[0,T]\}$. 
Remark that \eqref{finelem} is an infinite system of stochastic equations. In practice one 
should ``truncate" this system to solve numerically a suitable finite system instead,  
and one should also estimate the error caused by the truncation.  We will study 
such a procedure and the corresponding error elsewhere.

Our aim in this paper is to show that for some well-chosen  functions $\psi$,  
the above finite elements scheme has a unique 
solution $u^h$ for every $h\neq 0$, and that 
 for a given 
integer $k\geq0$  there exist random fields 
$v^{(0)}$, $v^{(1)}$,...,$v^{(k)}$ and $r_k$, on $[0,T]\times\bG_h$, 
such that 
almost surely
\begin{equation}                                               \label{expansion}
U^h_t(x)=v^{(0)}_t(x)+\sum_{1\leq j\leq k}v^{(j)}_t(x)\frac{h^j}{j!}+ 
r^{h}_t(x), 
\quad\text{$t\in[0,T]$,  $x\in\mathbb G_h$,}
\end{equation}
where $v^{(0)}$,..., $v^{(k)}$ do not depend on $h$, and 
there is a constant $N$, independent of $h$, such that 
\begin{equation}                                                 \label{remainder}
E\sup_{t\leq T} |h|^d \sum_{x\in\mathbb G_h}|r^h_t(x)|^2
\leq N |h|^{2(k+1)} E {\mathfrak K}_m^2 
\end{equation}
for all $|h|\in (0,1]$  and some $m> \frac{d}{2}$. 

To write  \eqref{finelem} more explicitly as an equation for 
$(U^h_t(y))_{y\in\bG_h}$, we introduce the following notation: 
\begin{align}                  \label{defR}
&R^{\alpha\beta}_{\lambda}=(D_{\beta}\psi_{\lambda},D^{\ast}_{\alpha}\psi), 
\quad \alpha,\beta\in\{0,1,...,d\}, \nonumber \\
&R^{\beta}_{\lambda}=R^{0\beta}_{\lambda}:=(D_{\beta}\psi_{\lambda},\psi) , \quad
R_{\lambda}:=R^{00}_{\lambda}:=(\psi_{\lambda},\psi),
\quad\lambda\in\bG, 
\end{align}
where $\psi_{\lambda}:=\psi^1_{\lambda}$, 
and $\bG:=\bG_1$. 

\begin{lemma}    \label{lemR}
For $\alpha, \beta\in \{1, ..., d\}$ and $\lambda\in \bG$ we have:
\[ R^{\alpha\beta}_{-\lambda} =R^{\alpha\beta}_{\lambda}, \quad R^\beta_{-\lambda}=-R^\beta_{\lambda}, \quad R_{-\lambda}=R_{\lambda}.\] 
\end{lemma}
\begin{proof}
Since $\psi(-x)=\psi(x)$ we deduce that for any $\alpha\in \{1, ..., d\}$ 
we have $D_\alpha \psi(-x)= - D_\alpha \psi(x)$. 
Hence for any $\alpha, \beta\in \{1, ..., d\}$ and $\lambda\in \bG$, 
a change of variables yields
\begin{align*}
R_{-\lambda}^{\alpha\beta}&=\! \int_{\RR^d} \! 
D_\beta \psi (z+\lambda) D^{\ast}_\alpha \psi(z) dz 
=\!  \int_{\RR^d}\! D_\beta\psi(-z+\lambda)D^{\ast}_\alpha \psi(-z) dz \\
&= \!\int_{\RR^d}\! D_\beta\psi (z-\lambda) D^{\ast}_\alpha \psi(z) dz 
=R^{\alpha\beta}_\lambda, \\
R^\beta_{-\lambda}&= \! \int_{\RR^d} \! D_\beta \psi (-z+\lambda) \psi(-z) dz 
= -\! \int_{\RR^d} \! D_\beta \psi (z-\lambda)  \psi(z) dz = -R^\beta_\lambda,\\
R_{-\lambda} &= \! \int_{\RR^d} \! \psi (-z+\lambda)  \psi(-z) dz 
= \! \int_{\RR^d} \!  \psi (z-\lambda)  \psi(z) dz =R_\lambda;
\end{align*}
this concludes the proof. 
\end{proof}

To prove the existence of a unique $V_h$-valued solution to \eqref{finelem},  
and a suitable estimate for it, we need the following condition. 
\begin{assumption}                                           \label{assumption invertibility}
{\it There is a constant $\delta>0$ such that 
$$
\sum_{\lambda,\mu\in\bG}R_{\lambda-\mu}
z^{\lambda}z^{\mu}\geq \delta\sum_{\lambda\in\bG}|z^{\lambda}|^2,  \quad \mbox{\rm for all }  (z^{\lambda})_{\lambda\in \bG}\in\ell_2(\bG).
$$
}
\end{assumption}

\begin{remark}                                                             \label{remark_Rlambda}
{\em  Note that since $\psi\in H^1$ has compact support, 
there exists a constant $M$ such that 
\[ |R_\lambda^{\alpha,\beta}|\leq M \quad 
\mbox{\rm for } \alpha, \beta \in \{0, ..., d\} \; \mbox{\rm and }\, \lambda \in \bG.\] }
 \end{remark} 

\begin{remark}                                 \label{remark norms}
{\em Due to Assumption \ref{assumption invertibility} for $h\neq 0$, 
$u:=\sum_{y\in\bG_h}U(y)\psi^h_y$, 
$U=(U(y))_{y\in\bG_h}\in \ell_2(\bG_h)$ we 
have 
\begin{align}                  \label{eqdelta}
|u|_0^2=&\sum_{x,y\in\bG_h}U(x)U(y)(\psi^h_x,\psi^h_y)  \nonumber\\
=&\sum_{x,y\in\bG_h}R_{(x-y)/h}U(x)U(y)  |h|^d 
\geq\delta\sum_{x\in\bG_h}U^2(x)|h|^d 
=\delta  |U|^2_{0,h}  .
\end{align}
Clearly, since $\psi$ has compact support, only finitely many $\lambda \in \bG$ 
are such that $(\psi_\lambda,\psi)\neq 0$; hence 
$$
|u|_0^2 \leq\sum_{x,y\in\bG_h} |R_{(x-y)/h}|\, |U(x)U(y)| |h|^d 
 \leq N |h|^d  \sum_{x\in\bG_h}U^2(x)=N |U|^2_{0,h} , 
$$
where $N$ is a constant depending only on $\psi$. }
\end{remark}

By virtue of this remark 
for each $h\neq 0$  the linear mapping ${\bf\Phi}_h$ from $\ell_2(\bG_h)$ to 
$V_h\subset L_2(\RR^d)$, 
defined by 
$$
{\bf\Phi}_hU:=\sum_{x\in\bG_h}U(x)\psi^h_x
\quad 
\text{for $U=(U(x))_{x\in\bG_h}\in\ell_2(\bG_h)$}, 
$$
is a one-to-one linear operator such that 
the norms of $U$ and ${\bf\Phi}_hU$ are equivalent, with constants 
independent of $h$.  
In particular, $V_h$ is a closed subspace of $L_2(\RR^d)$. Moreover, since 
$D_i\psi$ has compact support,  \eqref{eqdelta} implies that 
$$
|D_i u|_0 \leq \frac{N}{|h|}\|u\|
\quad \text{for all $u\in V_h$, 
\quad $i\in\{1,2,...,d\}$}, 
$$
where $N$ is a constant depending only on $D_i\psi$ and $\delta$.  Hence for any $h>0$ 
\begin{equation}                                              \label{H1}
|u|_{1}\leq N(1+|h|^{-1})|u|_0 \quad \text{for all  $u\in V_h$}
\end{equation}
with a constant $N=N(\psi,d,\delta)$ which does not depend on $h$. 
\begin{theorem}                                        \label{theorem eu}                                                 
Let Assumptions \ref{assumption coeff}  
through \ref{assumption invertibility} hold  with $m=0$. 
Then for each $h\neq 0$ 
equation \eqref{finelem} has a unique $V_h$-solution $u^h$. 
Moreover, there is a constant $N=N(d,K,\kappa)$ independent of $h$ such that 
\begin{align}                     \label{estimate1}
& E\sup_{t\in[0,T]}|u^h_t|_0^2+E\int_0^T|u^h_t|^2_{1}\,dt \nonumber \\                                            
&\qquad  \leq NE|\pi^h  \phi |_0^2+NE\int_0^T\big( |\pi^hf_s|_0^2
+\sum_{\rho}|\pi^h g^{\rho}_s|_0^2\big) \,ds\leq NE\cK_0^2
\end{align}
for all $h\neq 0$, where 
$\pi^h$ denotes the orthogonal projection  of $H^0=L_2$ 
into $V_h$. 
\end{theorem}
\begin{proof} We fix  $h\neq 0$ and define the bilinear forms 
$A^h$ and $B^{h\rho}$
by 
\begin{eqnarray*}
A^h_s(u,v)&:=&(a^{ij}_sD_{j}u,D^{\ast}_{i}v)
 +(b^i_sD_iu+c_su,v)\\
B^{h \rho}_s  (u,v)&:=&(\sigma^{i\rho}_sD_iu+ \nu_s^{\rho} u ,v)
 \end{eqnarray*}
 for all $u,v\in V_h$.  
 Using Assumption \ref{assumption coeff} with $m=0$, by virtue of 
 \eqref{H1} we have a constant $C=C(|h|,K,d,\delta, \psi)$, such that 
 \[ 
 A^h_s(u,v)\leq C|u|_0 |v|_0
 \quad
B^{h \rho}_s(u,v)\leq  C|u|_0 |v|_0\quad \text{for all $u,v\in V_h$}.  
\]
Hence, identifying  $V_h$ with its  dual space  
$(V_h)^{\ast}$  by the 
 help of the $L_2(\RR^d)$ inner product in $V_h$, we can see there exist 
 bounded linear operators 
 $\bA^h_s$ and $\bB^{h\rho}_s$  on $V_h$ such that 
 $$
A_s^h(u,v)=(\bA^h_su,v),\quad  B^{h\rho}_s(u,v)=(\bB^{h\rho}_su,v)
\quad \text{for all $u,v\in V_h$},
$$
and for all $\omega\in\Omega$ and $t\in[0,T]$. 
Thus \eqref{finelem} can be rewritten as 
\begin{equation}
u^h_t=\pi^h\phi  +\int_0^t(\bA_s^hu^h_s+\pi^hf_s)\,ds
+\int_0^t(\bB_s^{h\rho}u^h_s+\pi^hg^{\rho}_s)\,dW_s^{\rho}, 
\end{equation}
which is an (affine) linear SDE in the Hilbert space $V_h$. Hence, by classical 
results on solvability of SDEs with Lipschitz continuous coefficients in Hilbert  
spaces we get a unique $V_h$-solution $u^h=(u^h_t)_{t\in[0,T]}$. 
To prove estimate \eqref{estimate1} we may assume $E\cK_0^2<\infty$. 
By applying  It\^o's formula to $|u^h|_0^2$ 
we obtain 
\begin{equation}                                                             \label{Ito}
 |u^h(t)|_0^2=|\pi^h\phi|_0^2+
 \int_0^t I^h_s\,ds+\int_0^tJ^{h,\rho}_s\,dW_s^{\rho}, 
\end{equation}
with 
\begin{align*}                                                                       
I^h_s:=&2(\bA_s^hu^h_s+\pi^hf_s,u^h_s)
+\sum_{\rho}|\bB_s^{h\rho}u^h_s+\pi^hg^{\rho}_s|_0^2 , \\                                           
J^{h\rho}_s:=&
 2(\bB_s^{h\rho}u^h_s+\pi^hg^{\rho}_s,u^{h}_s). 
 \end{align*}
 Owing to  Assumptions \ref{assumption coeff}, 
 \ref{assumption data} and \ref{assumption parab},  by 
 standard estimates we get
\begin{equation}                                                        \label{Ih}
 I^h_s\leq -\kappa |u^h(s)|^2_{1}+N\Big( |u^h_s|_0^2+|f_s|_0^2
 +\sum_{\rho}|g^{\rho}_s|_0^2 \Big)
\end{equation}
 with a constant $N=N(K,\kappa,d)$; thus from \eqref{Ito} using Gronwall's lemma 
 we obtain 
\begin{equation}                          \label{estimate0}
 E|u^h_t|_0^2+\kappa E\int_0^T|u^h_s|^2_{1}\,ds
 \leq NE\cK_0^2\quad t\in[0,T]
\end{equation}
 with a constant $N=N(T,K,\kappa,d)$. 
One  can estimate $E\sup_{t\leq T} |u^h_t|_0^2$ also in a standard way. 
Namely,  since 
$$
\sum_{\rho}|J^{h\rho}_s|^2
\leq N^2 \, \big(|u^h_s|_1^2+ |g_s|_0^2 \big) \,  \sup_{s\in[0,T]}|u^h_s|_0^2
$$
with a constant $N=N(K,d)$, by the Davis inequality we have 
\begin{align} 
 E\sup_{t\leq T}\Big|\int_0^tJ^h_s\,dW_s^{\rho}\Big| 
 &\leq \, 3E\Big(\int_0^T\sum_{\rho}|J^{h,\rho}_s|^2\,ds\Big)^{1/2}  \nonumber \\
 & \leq \, 3NE\Big(\sup_{s\in[0,T]}|u^h_s|_0^2\int_0^T\big( |u^h_s|^2_1+|  g_s |_0^2\big) \,ds\Big)^{1/2} \nonumber \\
                                   \label{estimate3}
&\leq  \, \frac{1}{2}E\sup_{s\in[0,T]}|u^h_s|_0^2
+5N^2E\int_0^T \big( |u_s^h  |^2_1+| g_s |_0^2\big) \,ds. 
\end{align}
Taking supremum in $t$ in both sides of \eqref{Ito} and then 
using \eqref{Ih}, \eqref{estimate0} and \eqref{estimate3},  
we obtain estimate \eqref{estimate1}. 
\end{proof}

\begin{remark}                                                       \label{rku-h}
An easy computation using the symmetry of $\psi$ imposed 
in \eqref{sym_psi_Lambda} shows that for every $x\in \RR^d$
and $h\neq 0$ we have $\psi^{-h}_x=\psi^h_x$. Hence the uniqueness of the solution to \eqref{finelem} proved in Theorem \ref{theorem eu}
implies that the processes $u^{-h}_t$ and $u^h_t$ agree for $t\in [0,T]$ a.s.
\end{remark} 

To prove rate of convergence results we introduce more conditions 
on $\psi$ and $\Lambda$.\smallskip

\noindent{\bf Notation.}  Let $\Gamma$ denote the set 
of vectors $\lambda$ in $\bG$ such that 
the intersection of the support of $\psi_{\lambda}:=\psi^1_{\lambda}$ 
and the support of $\psi$ has positive 
Lebesgue measure in $\bR^d$. Then $\Gamma$ is a finite set.

\begin{assumption}                                                             \label{compatibility}
{\em Let $R_\lambda$, $R^i_\lambda$ and $R^{ij}_\lambda$ 
be defined by \eqref{defR}; then for 
$i,j,k,l\in\{1,2,...,d\}$:
\begin{align}
&\sum_{\lambda\in\Gamma}R_{\lambda}=1, 
\quad  \sum_{\lambda\in\Gamma}R_{\lambda}^{ij}=0, \label{2.5.1}
\\
&\sum_{\lambda\in\Gamma}\lambda_k R^{i}_{\lambda}=\delta_{i,k}, 
\label{2.5.2}
\\
& \sum_{\lambda\in\Gamma}\lambda_k \lambda_l R^{ij}_{\lambda}
=\delta_{\{i,j\},\{k,l\}} \quad 
\mbox{\rm for } \, i\neq j,  \quad 
\sum_{\lambda\in\Gamma}\lambda_k\lambda_l R^{ii}_{\lambda}
= 2  \delta_{\{i,i\},\{k,l\}},  
 \label{2.5.3.1}\\
&\sum_{\lambda\in\Gamma}Q^{ij,kl}_{\lambda}=0 
\quad  \mbox{\rm and }\quad 
\sum_{\lambda\in\Gamma}\tilde{Q}^{i,k}_{\lambda}=0,
\label{2.5.4}
\end{align}
where 
\[
Q_\lambda^{ij,kl}:=\int_{\bR^d}z_k z_l D_j\psi_{\lambda}(z) D_i^{\ast}\psi(z)\,dz, \quad 
\tilde{Q}_\lambda^{i,k}:= \int_{\bR^d}z_k D_i\psi_{\lambda}(z) \psi(z) \,dz, 
\]
and for sets of indices $A$ and $B$ the notation $\delta_{A,B}$ means $1$ when $A=B$ 
and  
$0$ otherwise.  }
\end{assumption}                                                

Note that if Assumption \ref{compatibility} holds true, 
then for  any family of real numbers 
 $X_{ij,kl}, i,j,k,l\in \{1, ..., d\}$ such that $X_{ij,kl}=X_{ji,kl}$  we deduce from the identities \eqref{2.5.3.1} 
that 
 \begin{equation}    \label{2.5.3}
\frac{1}{2}\sum_{i,j=1}^d  \sum_{k,l=1}^d X_{ij,kl} 
\sum_{\lambda\in\Gamma}\lambda_k \lambda_l R^{ij}_{\lambda} =\sum_{i,j=1}^d X_{ij,ij}.
\end{equation}  

Our main result reads as follows. 

\begin{theorem}                                                \label{theorem main2}
Let $J\geq0$ be an integer. 
Let Assumptions \ref{assumption coeff} 
and \ref{assumption data}  hold with 
 $m>2J+\frac{d}{2}+2$. Assume also Assumption \ref{assumption parab} 
and Assumptions \ref{assumption invertibility} and \ref{compatibility} 
on $\psi$ and $\Lambda$.  
Then expansion \eqref{expansion} and estimate \eqref{remainder} hold 
with a constant $N=N(m, J,\kappa,K,d,\psi,\Lambda)$, where $v^{(0)}=u$ is the 
solution of  \eqref{SPDE1} with initial condition $\phi$ in \eqref{ini}. 
Moreover, in the expansion \eqref{expansion} we have $v^{(j)}_t=0$ for odd values of $j$.
\end{theorem}

Set 
$$
\bar u^h_t(x)=\sum_{j=0}^{\bar{J}} c_ju^{h/2^j}_t(x)
\quad t\in[0,T], \quad x\in\bG_h,  
$$
with $\bar{J}:= \lfloor \frac{J}{2} \rfloor$,  
$(c_0,..,c_{\bar{J}}) =(1,0...,0)V^{-1}$, where $V^{-1}$ is the inverse of 
the $(\bar{J}+1)\times(\bar{J}+1)$ Vandermonde matrix 
$$
V^{ij}=2^{-4(i-1)(j-1)},\quad i,j=1,2,...,\bar{J}+1. 
$$ 
We make also the following assumption. 
\begin{assumption}                             \label{assumption uU}
$$
\psi(0)=1\quad\text{and $\psi(\lambda)=0$ for $\lambda\in\bG\setminus\{0\}$.}
$$
\end{assumption}

\begin{corollary}                                        \label{corollary1}
{\it Let Assumption \ref{assumption uU} and the assumptions 
of Theorem \ref{theorem main2} hold. 
Then 
$$
E\sup_{t\in[0,T]}
\sum_{x\in\mathbb G_h}|u_t(x)-  \bar{u}^h_t(x)  |^2 |h|^{d} 
\leq 
|h|^{2J+2} N E  {\mathfrak K}_m^2 
$$
for $|h|\in (0,1]$, with a constant $N=N(m,K,\kappa,J,T,d,\psi,\Lambda)$   
independent of $h$, where 
$u$ is the solution of \eqref{SPDE1}-\eqref{ini}. }
\end{corollary}

\section{Preliminaries} 

Assumptions \ref{assumption coeff}, \ref{assumption data} 
and \ref{assumption invertibility} 
are assumed to hold throughout this section. 
Recall that  $|\cdot|_{0,h}$ denote the norm,  
and $(\cdot,\cdot)_{0,h}$ denote the inner product in $\ell_2(\bG_h)$, i.e.,
 $$
|\varphi_1|^2_{0,h}:=|h|^d \sum_{x\in\bG_h}  \varphi_1^2(x)\, ,
 \quad (\varphi_1,\varphi_2)_{0,h}:=|h|^d \sum_{x\in\bG_h}\varphi_1(x)\varphi_2(x)
$$
for functions $\varphi_1, \varphi_2 \in\ell_2(\bG_h)$.  

 Dividing by {$|h|^d$, it is easy to see  
 that  the equation \eqref{finelem}  for the finite elements 
approximation 
$$
u^h_t(y)=\sum_{x\in\bG_h}U^h_t(x)\psi_x(y), \quad t\in[0,T],\,y\in\bR^d, 
$$ 
can be rewritten  for $(U^h_t(x))_{x\in\bG_h}$ as 
\begin{align}                                                          \label{eqU}
\cI^hU^h_t(x)
= \, &\phi^{h}(x)+\int_0^t \big( \cL^h_sU^h_s(x)+f^h_s(x) \big)\,ds                   \nonumber\\
&+\int_0^t \big( \cM^{h,\rho}_sU^h_s(x)+g^{h,\rho}_s(x) \big)\,dW^{\rho}_s, 
\end{align}
$t\in[0,T],\,x\in\bG_h$, 
where
\begin{align}                                                                                       \label{smoothing}
\phi^h(x)&=\int_{\bR^d}\phi(x+hz)\psi(z)\,dz, 
\quad
f^h_t(x)=\int_{\bR^d}f_t(x+hz)\psi(z)\,dz  \nonumber\\
g^{h,\rho}_t(x)&=\int_{\bR^d}g^{\rho}_t(x+hz)\psi(z)\,dz, 
\end{align}
and for functions $\varphi$ on $\bR^d$ 
\begin{align}
\cI^h\varphi(x)&=\sum_{\lambda\in\Gamma}
R_{\lambda}\varphi(x+h\lambda), \label{eqIh} \\
\cL^h\varphi(x)
& =
\sum_{\lambda\in\Gamma} \Big[ \frac{1}{h^2}A_t^h(\lambda,x)     +
\frac{1}{h}B_t^h(\lambda,x) + C_t^h(\lambda,x) \Big] \varphi(x+h\lambda),   \label{eqLh}\\
\cM^{h,\rho}\varphi(x)
& =\sum_{\lambda\in\Gamma} \Big[ \frac{1}{h}
S^{h,\rho}_t(\lambda, x) +N^{h,\rho}_t(\lambda,x) \Big] \varphi(x+h\lambda) ,  \label{eqMh}
\end{align}
with 
\begin{align*}
A_t^h(\lambda,x)& =\int_{\bR^d}a^{ij}_t(x+hz)D_j\psi_{\lambda}(z)D_i^{\ast}\psi(z)\,dz, 
\\
B_t^h(\lambda,x)&=\int_{\bR^d}b^{i}_t(x+hz)D_i\psi_{\lambda}(z)\psi(z)\,dz, \quad 
C_t^h(\lambda,x)=\int_{\bR^d}c_t(x+hz)\psi_{\lambda}(z)\psi(z)\,dz, 
\\
S_t^{h,\rho}(\lambda,x)&=\int_{\bR^d}\sigma^{i\rho}_t(x+hz)D_i\psi_{\lambda}(z)\psi(z)\,dz, 
\quad
N_t^{h,\rho}(\lambda,x)=\int_{\bR^d}\nu^{\rho}_t(x+hz)\psi_{\lambda}(z)\psi(z)\,dz. 
\end{align*}

\begin{remark}				\label{rkU-h}
Notice that due to the symmetry of $\psi$ and $\Lambda$ 
required in \eqref{sym_psi_Lambda}, equation \eqref{eqU} is
invariant under the change of $h$ to $-h$.
\end{remark}

\begin{remark}                                              \label{remark uU}
{\it Recall the definition of $\Gamma$ 
introduced before Assumption \ref{compatibility}. Clearly
$$
R_{\lambda}=0, \,\,A^h_t(\lambda,x)=B^h_t(\lambda,x)
=C_t^h(\lambda,x)=S^{h,  \rho} _t(\lambda, x)=N^{h ,  \rho  }_t(\lambda,x)=0 
\quad \text{for $\lambda\in\bG\setminus\Gamma$}, 
$$ 
i.e., the definition of $\cI^h$, $\cL^h_t $ and $\cM^{h,\rho}_t $ does not change if 
the summation there is taken over 
$\lambda\in\bG$. Owing to Assumption \ref{assumption coeff} 
with $m=0$ and the bounds on $R^{\alpha\beta}_\lambda$, the  
operators  $\cL^h_t$ and $\cM^{h,\rho}_t$  are bounded linear operators 
on $\ell_2(\bG_h)$ such that for each $h\neq 0$ 
 and $t\in [0,T]$ 
$$
|\cL^h_t\varphi|_{0,h}\leq  N_h  |\varphi|_{0,h}, 
\quad 
\sum_{\rho}|\cM^{h,\rho}_t\varphi|^2_{0,h}\leq  N_h^2   |\varphi|^2_{0,h}
$$
for all $\varphi\in\ell_2(\bG_h)$, with a constant $  N_h  =N(|h|,K,d,\psi,\Lambda)$.  
One can similarly show that  
\begin{equation}                                                                       \label{I}
|\cI^h\varphi|_{0,h}\leq N|\varphi|_{0,h} \quad\text{for $\varphi\in\ell_2(\bG_h)$,}
\end{equation}
with a constant $N=N(K,d,\Lambda,\psi)$ independent of $h$. 
It is also easy to see that for every $\phi\in L_2$ and  $\phi^h$ 
defined as in \eqref{smoothing} we have 
$$
|\phi^h|_{0,h}\leq N|\phi|_{L_2}
$$
with a constant $N=N(d, \Lambda,\psi)$  which does not depend on $h$; therefore  
$$
|\phi^h|^2_{0,h}
+\int_0^T  \Big( |f_t^h|^2_{0,h}
+\sum_{\rho}|g_t^{h,\rho}|^2_{0,h}\Big)  \,dt\leq N^2\cK_0^2.
$$
}
\end{remark}
\begin{lemma}                                              \label{remark R}
 The inequality  \eqref{I} implies that the mapping $\cI^h$ is a bounded linear operator on $\ell_2(\bG_h)$. 
Owing to Assumption \ref{assumption invertibility} 
it has an inverse $(\cI^h)^{-1}$ on $\ell_2(\bG_h)$, 
and 
\begin{equation}                                      \label{inverse}
|(\cI^h)^{-1}\varphi|_{0,h}\leq\frac{1}{\delta} |\varphi|_{0,h}\quad\text{for $\varphi\in\ell_2(\bG_h)$. } 
\end{equation}
\end{lemma}
\begin{proof}
For $\varphi \in \ell_2(\bG_h)$ and $h \neq 0$ we have
\begin{align*} 
(\varphi,\cI^h\varphi)_{0,h}=&\,|h|^d \sum_{x\in\bG_h}\varphi(x)\cI^h\varphi(x)
= |h|^d \sum_{x\in\bG_h}\sum_{\lambda\in\bG}
\varphi(x)(\psi_{\lambda},\psi)\varphi(x+h\lambda)\\
=& \, |h|^d \sum_{x\in\bG_h}\sum_{y-x\in\bG_h}
\varphi(x)(\psi_{\frac{y-x}{h}},\psi)\varphi(y)
= |h|^d \sum_{\lambda,\mu\in\bG}
\varphi(h\mu)R_{\lambda-\mu}\varphi(h\lambda)
\\
\geq &\,  \delta |h|^d \sum_{\lambda\in\bG}|\varphi(h\lambda)|^2=\delta|\varphi|^2_{0,h}. 
\end{align*} 
Together with \eqref{I}, this estimate implies that $\cI^h$ is invertible and that \eqref{inverse} holds. 
\end{proof}

 Remark \ref{remark uU} and Lemma \ref{remark R} imply that equation \eqref{eqU} is an 
(affine) linear SDE in the Hilbert space $\ell_2(\bG_h)$, and by well-known results 
on solvability of SDEs with Lipschitz continuous coefficients in Hilbert spaces, equation 
\eqref{eqU} has a unique $\ell_2(\bG_h)$-valued continuons solution $(U_t)_{t\in[0,T]}$, 
which we call an {\it $\ell_2$-solution} to \eqref{eqU}. 

Now we formulate the relationship between 
equations \eqref{finelem} and \eqref{eqU}. 
\begin{theorem}                                                    \label{theorem equ}
Let Assumption \ref{assumption invertibility} hold. 
Then the following statements are valid. 

(i) Let Assumptions \ref{assumption coeff} and \ref{assumption data} 
be satisfied with $m=0$, and 
\begin{equation}                               \label{uU}
u^h_t=\sum_{x\in\bG_h}U^h_t(x)\psi^h_x,\quad t\in[0,T]
\end{equation}
be the unique $V_h$-solution of \eqref{finelem};  then $(U^h_t)_{t\in[0,T]}$ 
is the unique $\ell_2$-solution of \eqref{eqU}. 

(ii) Let Assumption \ref{assumption coeff} 
hold with $m=0$. Let $\Phi$ be an 
$\ell_2(\bG_h)$-valued $\cF_0$-measurable random variable, and let 
$F=(F_t)_{t\in[0,T]}$ and $G^{\rho}=(G^{\rho}_t)_{[0,T]}$ be $\ell_2(\bG_h)$-valued 
adapted processes such that almost surely 
$$
\cK^2_{0,h}
:=|\Phi|^2_{0,h}
+\int_0^T  \Big(  |F_t|^2_{0,h}+\sum_{\rho}|G_t^{\rho}|^2_{0,h}  \Big)  \,dt<\infty. 
$$
Then equation \eqref{eqU} with $\Phi$, $F$ and $G^{\rho}$ in place of 
$\phi^h$, $f^h$ and $g^{\rho,h}$, respectively, has a unique 
$\ell_2(\bG_h)$-solution $U^h=(U^h_t)_{t\in[0,T]}$. Moreover, 
if Assumption \ref{assumption parab} also holds then 
\begin{equation}                                                      \label{estimate U}
E\sup_{t\in[0,T]}|U_t^h|^2_{0,h}\leq NE\cK^2_{0,h}
\end{equation}
with a constant   $N=N(K,d,\kappa,\delta,\Lambda,  \psi  )$ which does not depend on $h$. 
\end{theorem}
\begin{proof} 
(i) Substituting \eqref{uU} into equation \eqref{finelem}, then dividing both  
sides of the equation by  $|h|^d$  we obtain 
equation \eqref{eqU} for  $U^h$ by simple calculation. 
Hence by 
Remark \ref{remark uU} we can see that  $U^h$  is the unique  $\ell_2(\bG)$-valued 
solution to \eqref{eqU}. 

To prove (ii) we use  Remark \ref{remark R} on the invertibility 
of $\cI^h$ and a standard result on solvability of SDEs in Hilbert spaces to see that 
equation \eqref{eqU} with $\Phi$, $F$ and $G^{\rho}$ has a unique $\ell_2(\bG)$-valued 
solution  $U^h$. We claim that $u^h_t(\cdot)=\sum_{y\in\bG}  U^h_t  (y)\psi^h_{y}(\cdot)$ is the $V_h$-valued 
solution of \eqref{finelem} with 
$$
\phi(\cdot):=\sum_{y\in\bG_h}(\cI^{h})^{-1}\Phi(y)\psi^h_y(\cdot), \quad 
f_t(\cdot):=\sum_{y\in\bG_h}(\cI^h)^{-1}F_t(y)\psi^h_y(\cdot),
$$
and
$$ 
g_t^{\rho}(\cdot):=\sum_{y\in\bG_h}(\cI^h)^{-1}G^{\rho}_t(y)\psi^h_y(\cdot), 
$$
respectively. Indeed, \eqref{eqIh} yields
\begin{align*}
|h|^{-d} ( \phi , \psi^h_x)&= |h|^{-d} \sum_{y\in\bG_h}(\psi^h_y,\psi^h_x)(\cI^h)^{-1}\Phi(y)
=\sum_{y\in\bG_h}R_{\frac{y-x}{h}}(\cI^h)^{-1}\Phi(y)  \\
& =\cI^h\{(\cI^h)^{-1}\Phi\}(x)= \Phi(x), \quad x\in\bG_h. 
\end{align*}
In the same way we have 
$$
|h|^{-d}  (f_t,\psi_{x}^h)=F_t(x), \quad  |h|^{-d} (g^{\rho}_t,\psi_{x}^h)=G^{\rho}_t(x)\quad 
\text{for $x\in\bG_h$}, 
$$
which proves the claim. Using Remarks \ref{remark norms} and \ref{remark R}  we deduce 
\begin{align*} 
\|\phi\|& \leq N|(\cI^{h})^{-1}\Phi|_{0,h}\leq \frac{N}{\delta}| \Phi |_{0,h}, 
\quad 
\|f_t\|\leq N|(\cI^{h})^{-1}F_t|_{0,h}\leq \frac{N}{\delta}|F_t|_{0,h}, \\
\sum_{\rho}\|g^{\rho}_t\|^2 &
\leq N^2\; \sum_{\rho}  |(\cI^{h})^{-1}G^{\rho}_t|^2_{0,h}
\leq \frac{N^2}{\delta^2}\sum_{\rho}  |G^{\rho}_t|_{0,h}^2 
\end{align*} 
with a constant $N=N(\psi,\Lambda)$. 
Hence by Theorem \ref{theorem eu} 
$$
E\sup_{t\leq T}\|u_t^h\|^2\leq NE|\phi|^2_{0,h}
+NE\int_0^T\Big( |F_t|^2_{0,h}+\sum_{\rho}|G^{\rho}_t|^2_{0,h}\Big) \,dt
$$
with $N=N(K,T,\kappa,d,\psi,\Lambda,\delta)$  independent of $h$, 
which by virtue of Remark 
\ref{remark norms} implies 
estimate \eqref{estimate U}. 
\end{proof}

\section{Coefficients of the expansion}

Notice that the lattice $\bG_h$ and the space $V_h$ can be ``shifted" to 
any $y\in\bR^d$, i.e., we can consider $\bG_h(y):=\bG_h+y$ and 
$$
V_h(y)
:=\Big\{\sum_{x\in\bG_h(y)}U(x) \psi^{h}_{x}: (U(x))_{x\in{\bG_h(y)}}\in \ell_2(\bG_h(y))\Big\}.  
$$
Thus equation \eqref{finelem} for $u^h=\sum_{x\in\bG_h(y)}U(x)\psi^h_x$ 
should be satisfied for $\psi_x$, $x\in\bG_h(y)$. Correspondingly, 
equation \eqref{eqU} can be considered 
for all $x\in \bR^d$ 
instead of $x\in \bG_h$.   

To determine the coefficients $(v^{(j)})_{j=1}^{k}$  in  the  expansion \eqref{expansion} 
we differentiate formally \eqref{eqU} in the parameter $h$, $j$ times, for $j=1,2,...,J$, 
and consider the system of SPDEs we obtain for the formal derivatives 
 \begin{equation} 			\label{Djh}
u^{(j)}=D_h^jU^h\big|_{h=0},  
\end{equation}  
where $D_h$ denotes differentiation in $h$. 
To this end given an integer $n\geq 1$ let us  
first investigate the operators $\cI^{(n)}$, $\cL^{(n)}_t$ and  $\cM^{(n)\rho}_t$ 
defined by 
\begin{align}			\label{I(n)}
&\cI^{(n)}\varphi(x)=D^n_h\cI^h\varphi(x)\big|_{h=0}, 
\quad 
\cL^{(n)}_t\varphi(x)=D^n_h\cL^h_t\varphi(x)\big|_{h=0}, \nonumber \\
& \cM^{(n)\rho}_t\varphi(x)=D^n_h\cM^{h,\rho}_t\varphi(x)\big|_{h=0}
\end{align}
for $\varphi\in C_0^{\infty}$.

\begin{lemma}                                                               \label{lemma hdiff}
Let Assumption \ref{assumption coeff} hold with $m\geq n+l+2$ for 
nonnegative integers $l$ and $n$. Let Assumption \ref{compatibility}. 
also hold. Then for $\varphi\in C_0^{\infty}$ and $t\in [0,T]$ we have 
\begin{equation}                                 \label{estimate_hdiff}
|\cI^{(n)}\varphi|_l\leq  N|\varphi|_{l+n}, \quad 
|\cL^{(n)}_t\varphi|_{l} \leq N|\varphi|_{l+2+n}, \quad
\sum_\rho|\cM^{(n)\rho}_t\varphi|_{l}^2 \leq N^2|\varphi|_{l+1+n}
\end{equation}
with a constant $N=N(K,d,l, n,\Lambda,\Psi)$  which does not depend on $h$. 
\end{lemma}
\begin{proof}
Clearly, $\cI^{(n)}
=\sum_{\lambda\in\Gamma}
R_{\lambda}\partial^n_{\lambda}\varphi$, 
where 
\begin{equation}  \label{deltalambda}
\partial_{\lambda}\varphi:=\sum_{i=1}^d\lambda^iD_i  \varphi.
\end{equation}  
This shows the existence 
of a constant $N=N(\Lambda, \psi,d,n)$ 
such that the first estimate in \eqref{estimate_hdiff} 
holds. To prove the second estimate we first claim 
the existence of a constant $N=N(K,d,l,\Lambda,\psi )$  
such that 
\begin{equation}                  \label{eq4.3}
\Big|D^n_h  \Phi_t  (h,\cdot)\big|_{h=0}\Big|_l\leq N|\varphi|_{l+n+2}
\end{equation}
for 
$$
\Phi_t  (h,x):=h^{-2}\sum_{\lambda\in\Gamma}\varphi(x+h\lambda)
\int_{\bR^d}a^{ij}_t(x+hz)D_j\psi_{\lambda}(z)D_i^{\ast}\psi(z)\,dz. 
$$ 
 Recall the definition of $R^{ij}_\lambda$ given in \eqref{defR}.  
 To prove  \eqref{eq4.3}  we write 
$\Phi_t(h,x)=\sum_{i=1}^3 \Phi^{(i)}_t(h,x)$ for $h\neq0$ 
with 
\begin{align*}
\Phi^{(1)}_t(h,x)& =h^{-2}\sum_{\lambda\in\Gamma}\varphi(x+h\lambda)
\int_{\bR^d}a^{ij}_t(x)D_j\psi_{\lambda}(z)D_i^{\ast}\psi(z)\,dz
\\
& =h^{-2}a^{ij}_t(x)\sum_{\lambda\in\Gamma}\varphi(x+h\lambda)R_{\lambda}^{ij}, 
\\
\Phi^{(2)}_t(h,x)&=h^{-1}\sum_{\lambda\in\Gamma}\varphi(x+h\lambda)
\int_{\bR^d} \sum_{k=1}^d  D_k a^{ij}_t(x) z_k D_j\psi_{\lambda}(z)D_i^{\ast}\psi(z)\,dz, 
\\
&=h^{-1}\sum_{\lambda\in\Gamma}\varphi(x+h\lambda)
Da^{ij}_t(x)S^{ij}_{\lambda}, 
\end{align*}
for 
\[ S^{ij}_{\lambda}:=\int_{\bR^d}zD_j\psi_{\lambda}(z)D_i^{\ast}\psi(z)\,dz  \in {\RR^d} ,\]
and
\begin{align*}
\Phi^{(3)}_t(h,x)&=\sum_{\lambda\in\Gamma}\varphi(x+h\lambda)
\int_{\bR^d}
\int_0^1(1-\vartheta)D_{kl}a^{ij}_t(x+h\vartheta z)z^kz^lD_j\psi_{\lambda}(z)D_i^{\ast}\psi(z)
\,d\vartheta\,dz, 
\end{align*}
where $D_{kl}:=D_kD_l$. 
Here we used Taylor's formula
\begin{equation}                                           \label{taylor} 
f(h)=\sum_{i=0}^{n}\frac{h^{i}}{i!}
f^{(i)}(0)+\frac{h^{n+1}}{n!}
\int_{0}^{1}(1-\theta)^{n}f^{(n+1)}(h\theta)
\,d\theta
\end{equation}
with $n=1$  and $f(h):=a^{ij}_t(x+h\lambda)$. 

Note that   Lemma \ref{lemR} and \eqref{2.5.1} in Assumption \ref{compatibility}  imply
\begin{align}              \label{phi1}
\Phi^{(1)}_t(h,x)&=
\frac{1}{2}a^{ij}_t(x)\sum_{\lambda\in\Gamma}
R^{ij}_{\lambda}h^{-2}(\varphi(x+h\lambda)-2\varphi(x)+\varphi(x-h\lambda))
  \nonumber    \\                                                                  
& =\frac{1}{2}a^{ij}_t(x)\sum_{\lambda\in\Gamma}
R^{ij}_{\lambda}\int_{0}^{1}\int_{0}^{1}
 \partial^2_{ \lambda}
\varphi(x+h\lambda(\theta_{1}-\theta_{2})) \,d\theta_{1}d\theta_{2}. 
\end{align}
 To rewrite $\Phi^{(2)}_t(h,x)$ note that $S^{ij}_{-\lambda} = -S^{ij}_{\lambda}$; 
 indeed since $\psi(-x)=\psi(x)$ the  change of variables
$y=-z$ implies that
\begin{align}             \label{S-lambda}
S^{ij}_{-\lambda} &=\int_{\RR^d} z D_j\psi(z+\lambda) D_i^\ast \psi(z) dz 
=-\int_{\RR^d} y D_j\psi(-y+\lambda) D_i^\ast \psi(-y) dy
\nonumber \\
&=-\int_{\RR^d} y D_j\psi(y-\lambda) D_i^\ast \psi(y) dy=-S^{ij}_{\lambda}.
\end{align}
Furthermore, an obvious change of variables,  \eqref{S-lambda} 
and Lemma \ref{lemR}  yield 
\begin{align*}
S^{ji}_\lambda&=\int_{\RR^d} z D_i\psi(z-\lambda) D_j^\ast \psi(z) dz 
= \int_{\RR^d} (z+\lambda) D_i\psi(z) D_j^\ast \psi(z+\lambda) dz\\
&=\int_{\RR^d} z D_i^\ast \psi(z) D_j \psi_{-\lambda}(z) dz 
+ \lambda \int_{\RR^d} D_i^\ast \psi(z) D_j\psi_{-\lambda}(z) dz \\
&=S^{ij}_{-\lambda} +\lambda R^{ij}_{-\lambda} 
=-S^{ij}_{\lambda} +\lambda R^{ij}_{\lambda} .
\end{align*}
This implies
\[ S^{ji}_\lambda + S^{ij}_\lambda =\lambda R^{ij}_\lambda,  \quad  i,j=1, ..., d .\]
Note that since $a^{ij}_t(x)=a^{ji}_t(x)$, we deduce 
\begin{equation}           \label{S-R}
 D a^{ij}_t(x) S^{ij}_\lambda = D a^{ij}_t(x) S^{ji}_\lambda 
 =\frac{1}{2}  D a^{ij}_t(x) \lambda R^{ij}_\lambda 
= \frac{1}{2}  R^{ij}_\lambda   \partial_\lambda  a^{ij}_t(x),
\end{equation} 
for  $\partial_\lambda F $  defined by \eqref{deltalambda}. 
Thus the equations \eqref{S-lambda} and \eqref{S-R} imply 
\begin{align}            \label{phi2}     
\Phi^{(2)}_t(h,x)
& = \frac{1}{2}
\sum_{\lambda\in\Gamma}h^{-1}
(\varphi(x+h\lambda)-\varphi(x-h\lambda)) 
Da^{ij}_t(x)S^{ij}_{\lambda}
\nonumber \\                                                
& = \frac{1}{4} \sum_{\lambda\in\Gamma}  R^{ij}_{\lambda}\partial_{\lambda}a^{ij}_t(x)
\; 2 \int_{0}^{1}\partial_{\lambda}\varphi(x+h\lambda(2\theta-1))
\,d\theta .
\end{align}
From \eqref{phi1} and \eqref{phi2} we get 
\begin{align*}
D^n_h\Phi^{(1)}_t(h,x)\big|_{h=0}& =
\frac{1}{2}a^{ij}_t(x)\sum_{\lambda\in\Gamma}R^{ij}_{\lambda}  
\int_{0}^{1}\int_{0}^{1} \partial^{n+2}_{\lambda} {\varphi(x)}(\theta_1-\theta_2)^n
\,d\theta_{1}d\theta_{2}, \\
D^n_h\Phi^{(2)}_t(h,x)\big|_{h=0}
&=
\frac{1}{2} \sum_{\lambda\in\Gamma}R^{ij}_{\lambda}\partial_{\lambda}a^{ij}
\int_{0}^{1} \partial^{n+1}_{\lambda}{\varphi(x)} (2\theta-1)^n
\,d\theta.  
\end{align*}
Furthermore, the definition of $\Phi^{(3)}_t(h,x)$ yields
\begin{align*}
& D^n_h\Phi^{(3)}_t(h,x)\big|_{h=0}\\
&\quad =
\sum_{\lambda\in\Gamma}
\sum_{k=0}^n{n\choose k}\partial_{\lambda}^{n-k}\varphi{(x)}
\int_{\bR^d}
\int_0^1(1-\theta)\theta^{k}\partial_z^{k}D_{kl}a^{ij}_t(x)z^kz^lD_j\psi_{\lambda}(z)D_i^{\ast}\psi(z)
\,d\theta\,dz.  
\end{align*}
 Using Assumption \ref{assumption coeff} and Remark \ref{remark_Rlambda}, this completes the proof of \eqref{eq4.3}. 

Taylor's formula \eqref{taylor} with $n=0$ and $f(h):=b^i_t(x+h  z )$ implies
\begin{align*}
\tilde{\Phi}_t(h,x):&=h^{-1} \sum_{\lambda\in \Gamma} \varphi(x+h\lambda) \int_{\RR^d} b^i_t(x+hz) D_i\psi_\lambda(z) \psi(z) dz\\
&={\Phi}^{(4)}_t(h,x) + {\Phi}^{(5)}_t(h,x),
\end{align*}
with
\begin{align*}
{\Phi}^{(4)}_t(h,x)&= h^{-1} b^i_t(x) 
\sum_{\lambda\in \Gamma} \varphi(x+h\lambda) R^i_\lambda, \\
{\Phi}^{(5)}_t(h,x)&=  \sum_{\lambda\in \Gamma} \varphi(x+h\lambda) 
\int_{\RR^d} \int_0^1 (1-\theta) \sum_{k=1}^d 
D_k b^i_t(x+h\theta z) z_k 
D_i\psi_\lambda(z) \psi(z) d\theta dz.
\end{align*}
Using Lemma \ref{lemR} and computations similar to those used 
to prove \eqref{eq4.3} we deduce that
\begin{align*}
\Phi^{(4)}_t(h,x)&=\frac{1}{2} \sum_{\lambda\in \Gamma} 
h^{-1} \big[ \varphi(x+h\lambda)-\varphi(x-h\lambda)\big] b^i_t(x) R^i_\lambda\\
&= b^i_t(x) \sum_{\lambda\in \Gamma} R^i_\lambda 
\int_0^1 \partial_\lambda \varphi\big(x+ h\lambda(2\theta-1)\big) d\theta, 
\end{align*}
which yields
\[ D^n_h\Phi^{(4)}_t(h,x) \big|_{h=0} 
= b^i_t(x) \sum_{\lambda\in \Gamma} R^i_\lambda \partial_\lambda^{n+1}  
\varphi({x})  \int_0^1   (2\theta-1)^n d\theta.
\] 
Furthermore, the definition of $\Phi^{(5)}(h,x)$ implies
\begin{align*}
& D^n_h\Phi^{(5)}_t(h,x){\big|_{h=0}}\\
&\quad =\sum_{\lambda\in \Gamma} \sum_{\alpha=0}^n
 {n\choose \alpha } \int_{\bR^d}
\int_0^1(1-\theta) \partial_{\lambda}^{n-\alpha}\varphi({x})
\theta^{\alpha} \partial_z^{\alpha} D_{\alpha}b^{i}_t(x) z^\alpha D_i\psi_{\lambda}(z)\psi(z)
\,d\theta\,dz.  
\end{align*}
This implies 
the existence of a constant $N=N(K,d,l,\Lambda,\psi )$   which does not depend on $h$  
such that 
\begin{equation}                  \label{eq4.9}
\Big|D^n_h \tilde{ \Phi}_t  (h,\cdot)\big|_{h=0}\Big|_l\leq N|\varphi|_{l+n+1} .
\end{equation}
Finally, let 
\[ \Phi^{(6)}_t(h,x):=\sum_{\lambda \in \Gamma} \varphi(x+h\lambda)
\int_{\RR^d} c_t(x+hz) \psi_\lambda(z)\psi(z) dz.\]
Then we have
\[ D^n_h\Phi^{(6)}_t(h,x){\big|_{h=0}}
 =\sum_{\lambda\in \Gamma} \sum_{\alpha=0}^n
 {n\choose \alpha } \partial_{\lambda}^{n-\alpha}\varphi({x}) 
 \int_{\bR^d}
 \partial_z^{\alpha} c_t(x) \psi_{\lambda}(z)\psi(z)  \,dz. 
\] 
so that 
\begin{equation}                                                                            \label{eq4.10}
\Big|D^n_h  \Phi^{(6)}_t (h,\cdot)\big|_{h=0}\Big|_l\leq N|\varphi|_{l+n}
\end{equation}
for some constant $N$ as above. 

Since  $\cL_t^h \varphi(x)=\Phi_t(h,x)+\tilde{\Phi}_t(h,x)+\Phi^{(6)}_t(h,x)$, 
the inequalities \eqref{eq4.3}, \eqref{eq4.9} and \eqref{eq4.10} imply
that 
$\cL^{(n)}_t$ satisfies the estimate in \eqref{estimate_hdiff}; 
the upper estimates of  $\cM^{(n),\rho}_t$  can be proved similarly.
\end{proof}

For an integer $k\geq0$ define the operators $\hat L_t^{(k)h}$, $\hat M_t^{(k)h,\rho}$ 
and $\hat I^{(k)h}$ by 
\begin{align}
\hat L_t^{(k)h}\varphi=\cL^h_t\varphi-\sum_{i=0}^{k}\frac{h^i}{i!}& \cL_t^{(i)}\varphi,  
\quad
\hat M_t^{(k)h,\rho}\varphi=\cM^{h,\rho}_t\varphi-\sum_{i=0}^{k}\frac{h^i}{i!}\cM^{(i)\rho}_t\varphi,  \nonumber \\ 
&\hat I^{(k)h}\varphi=\cI^h\varphi-\sum_{i=0}^{k}\frac{h^i}{i!}\cI^{(i)}\varphi,  			\label{expand_LMIh}
\end{align}
where $\cL^{(0)}_t:=\cL_t$, $\cM^{(0),\rho}_t:=\cM^{\rho}_t$, and $\cI^{(0)}$ is the identity 
operator.

\begin{lemma}   \label{lemma4.2}
{\it Let Assumption \ref{assumption coeff} hold with $m\geq k+l+3$ for 
nonnegative integers $k$ and $n$. Let Assumption \ref{compatibility}
also hold. Then for $\varphi\in C_0^{\infty}$ we have 
\begin{align*}
& |\hat L_t^{(k)h}\varphi|_l\leq N |h|^{k+1} |\varphi|_{l+k+3},  
\quad
\sum_{\rho}|\hat M_t^{(k)h,\rho}\varphi|^2_l\leq N^2 |h|^{2k+2} |\varphi|^2_{l+k+2}. 
\\
& |\hat I^{(k)h}\varphi|_l\leq N  |h|^{k+1} |\varphi|_{k+1}, 
\end{align*}
 for a constant $N$ which does not depend on $h$. 
}
\end{lemma}
\begin{proof} 
We obtain the estimate for $\hat L_t^{(k)h}$ by applying Taylor's formula 
\eqref{taylor} 
to  $f(h):= \Phi^{(i)}_t(h,x)$ for $i=1,..., 6$ defined in the proof of Lemma \ref{lemma hdiff}, 
and by estimating the remainder term 
$$
\frac{h^{k+1}}{k!}
\int_{0}^{1}(1-\theta)^{k}f^{(k+1)}(h\theta)
\,d\theta
$$
using the Minkowski inequality. Recall that  $\cL_t\varphi(x)  = \cL^{(0)}_t\varphi(x)$.  Using Assumption \ref{compatibility}  we prove that 
$\cL_t^{(0)} \varphi(x)= \lim_{h\to 0} \cL_t^h \varphi(x)$.  
We have $\cL^{h}_t\varphi(x) = \sum_{i=1}^6 \Phi^{(i)}_t(h,x)$ for 
$h\neq 0$. 
 The proof of Lemma \ref{lemma hdiff} shows that   $\tilde{\Phi}^{(i)}_t(0,x)
:=\lim_{h\to 0} \Phi^{(i)}_t(h,x)$ exist and we identify these limits.   
Using   \eqref{phi1}, \eqref{phi2}  and  \eqref{2.5.3} 
with $X_{ij,kl}=a^{ij}_t(x) D_{kl}\varphi(x)$ 
(resp. $X_{ij,kl}=\partial_k a^{ij}_t(x) \partial_l \varphi(x)$)  we deduce 
\begin{align*}
 \tilde{\Phi}^{(1)}_t(0,x) =&\sum_{i,j} \frac{1}{2} a^{ij}_t(x) \sum_{k,l} D_k D_l \varphi(x) \sum_{\lambda\in \Gamma}
 \lambda_k\lambda_l R^{ij}_\lambda =
\sum_{i,j} a^{ij}_t(x) D_{ij}\varphi(x) \\
 \tilde{\Phi}^{(2)}_t(0,x)= & \frac{1}{2} \sum_{i,j} \sum_{k,l}  \partial_k a^{ij}_t(x) \partial_l \varphi(x) \sum_{\lambda\in \Gamma} 
\lambda_k \lambda_l R^{ij}_\lambda = \sum_{i,j}\partial_i a^{ij}_t(x) \partial_j\varphi(x),
\end{align*}
which implies that $\tilde{\Phi}^{(1)}_t(0,x)+\tilde{\Phi}^{(2)}_t(0,x) 
=D_i\big(a^{ij}_t D_j\varphi \big)(x)$. The first identity in \eqref{2.5.4} (resp.
\eqref{2.5.2}, the second identity in \eqref{2.5.4} and the first identity in \eqref{2.5.1}) imply
\begin{align*}
 \tilde{\Phi}^{(3)}_t(0,x) =&\frac{1}{2}\varphi(x) \sum_{k,l} \sum_{i,j} D_{kl}a^{ij}_t(x) 
 \sum_{\lambda\in \Gamma} Q^{ij,kl}_\lambda =0,\\
 \tilde{\Phi}^{(4)}_t(0,x) =&\sum_i b^i_t(x) \sum_k \partial_k\varphi(x)
 \sum_{\lambda\in \Gamma} R^i_\lambda \lambda_k=\sum_i
b^i_t(x) \partial_i\varphi(x),\\
 \tilde{\Phi}^{(5)}_t(0,x) =&  \frac{1}{2} \varphi(x)  \sum_k \sum_i  D_kb^i_t(x) 
 \sum_{\lambda\in \Gamma} \tilde{Q}^{i,k}_\lambda =0,\\
\tilde{\Phi}^{(6)}_t(0,x) =&\varphi(x) c_t(x) \sum_{\lambda\in \Gamma} R_\lambda 
= \varphi(x) c_t(x).
\end{align*}

This completes the identification of $\cL_t$ as the limit of $\cL_t^h$. 
Using once more the Minkowski inequality 
 and usual estimates, we prove the upper estimates of  the $H^l$ norm of 
$\hat{L}^{(k)h}_t\varphi$. 
The other estimates can be proved similarly. 
\end{proof}
Assume that Assumption \ref{assumption data} is satisfied  
with $m\geq J+1$ for an integer 
 $J\geq0$.   A simple computation made for  differentiable functions 
 in place of the formal ones introduced in \eqref{Djh} 
 shows the following identities 

\begin{align*}
\phi^{(i)}(x)=&
\int_{\bR^d}\partial_z^{i}\phi(x)\psi(z)\,dz, 
\;
f^{(i)}_t(x)=\int_{\bR^d}\partial_z^{i}f_t(x)\psi(z)\,dz,
\;
g_t^{(i)\rho}(x)=\int_{\bR^d}\partial_z^{i}g_t^{\rho}(x)\psi(z)\,dz, 
\end{align*} 
where  $\partial_z^i \varphi$ is defined iterating \eqref{deltalambda}, 
while  $\phi^h$, $f_t^h$ and $g_t^{h,\rho}$ are defined in \eqref{smoothing}. 
Set 
\begin{align}			\label{Taylor_free}
 \hat\phi^{(J)h} :=&\phi^{h}-\sum_{i=0}^{J}\frac{h^i}{i!}\phi^{(i)}, 
\;
 \hat f^{(J)h}_t :=f^{h}_t-\sum_{i=0}^{J}\frac{h^i}{i!}f_t^{(i)}  \; \mbox{\rm and }\, 
 \hat g^{(J)h\rho}_t :=g^{h,\rho}_t-\sum_{i=0}^{J}g^{(i)\rho}_t\frac{h^i}{i!}. 
\end{align}
\begin{lemma}		\label{lemma4.3}
Let Assumption \ref{assumption coeff} holds with $m\geq l+J+1$ for nonnegative 
integers $J$ and $l$. Then there is a constant 
$N=N(J,l,d,\psi)$  independent of $h$  such that 
$$
 |\hat\phi^{(J)h}|_l \leq  |h|^{J+1}N|\phi|_{l+1+J},
\quad 
 |\hat f^{(J)h}_t|_l \leq N |h|^{J+1} |f_t|_{l+1+J},
\quad
 |\hat g^{(J)h\rho}_t|_l \leq N  |h|^{J+1} |g^{\rho}_t|_{l+1+J}. 
$$
\begin{proof}
Clearly, it suffices to prove the estimate for 
$\hat\phi^{(J)h}$, and we may assume that $\phi\in C_0^{\infty}$. 
Applying Taylor formula \eqref{taylor} to $f(h)=\phi^h(x)$ for the remainder term 
we 
have 
$$
 \hat\phi^{(J)h}(x) =\frac{h^{J+1}}{J!}
\int_{0}^{1}\int_{\bR^d}(1-\theta)^{J}\partial_z^{J+1}\phi(x+\theta hz)\psi(z)\,dz. 
$$
Hence by Minkowski's inequality and the shift invariance of the Lebesgue measure 
we get 
$$
 |\hat\phi^{(J)h}(x)| \leq \frac{h^{J+1}}{J!}
\int_{0}^{1}\int_{\bR^d}(1-\theta)^{J}|\partial_z^{J+1}\phi(\cdot+\theta hz)|_l|\psi(z)|\,dz
\leq N   h^{J+1}  |\phi|_{l+J+1} 
$$
with a constant $N=N(J,m,d,\psi)$  which does not depend on $h$. 
\end{proof}
\end{lemma}

Differentiating formally  equation \eqref{eqU} with respect to $h$ at $0$ 
and using the definition of $\cI^{(i)}$ in \eqref{I(n)}, 
we obtain the following system of SPDEs: 
\begin{align}                                                       
& dv^{(i)}_t+\sum_{1\leq j\leq i}  {i\choose j} \cI^{(j)}dv^{(i-j)}_t
=
\Big\{\cL^{(0)}_tv^{(i)}_t
+f^{(i)}_t+\sum_{1\leq j\leq i} {i\choose j}  \cL^{(j)}_t v_t^{(i-j)} \Big\}\,dt 
\nonumber \\
&\qquad \qquad +\Big\{\cM^{(0)\rho}_tv^{(i)}_t+g^{(i)\rho}_t
+\sum_{1\leq j\leq i} {i\choose j} \cM^{(j)\rho}_t v_t^{(i-j)}\Big\}\,dW^{\rho}_t    ,          \label{SPDE system}  \\                                                                      
&v^{(i)}_0(x)=\phi^{(i)}(x) ,                                \label{system ini}                                                             
\end{align}
for $i=1,2,....,J$, $t\in[0,T]$ and $x\in\bR^d$, where 
$\cL^{(0)}_t=\cL_t$, $\cM^{(0)\rho}_t=\cM^{\rho}_t$, 
and $v^{(0)}=u$ is the solution to \eqref{SPDE1}-\eqref{ini}.  

\begin{theorem}                                                                               \label{theorem system}
Let Assumptions \ref{assumption coeff} and \ref{assumption data} 
hold with $m\geq J+1$ for an integer $J\geq1$.  Let Assumptions 
\ref{assumption parab} through \ref{compatibility} be also satisfied. 
Then \eqref{SPDE system}-\eqref{system ini} has a unique solution $(v^{(0)},...,v^{(J)})$ such that 
$v^{(n)}\in  \mathbb W^{m+1-n}_2(0,T)  $ for every $n=0,1,...,J$. Moreover, $v^{(n)}$ 
is a $H^{m-n}$-valued continuous adapted process, and for every $n=0,1,...,J$ 
\begin{equation}                                                                \label{systemsol}
E\sup_{t\leq T}|v^{(n)}_{t}|^{2}_{m-n}
+E\int_{0}^{T}|v^{(n)}_{t}|^{2}_{m+1-n}\,dt
\leq N  E\mathfrak{K}_m^2  
\end{equation}
with a constant $N=N(m,J,d,T,\Lambda,\psi,\kappa)$  independent of $h$,   and   $\mathfrak{K}_m$   defined in \eqref{gothic_Km}.  
\end{theorem}

\begin{proof} The proof is based on an induction argument.  
We can solve this system consecutively for $i=1,2,...,J$, 
by noticing that for each 
$i=1,2,...,k$ the equation 
for $v^{(i)}$ does not contain $v^{(n)}$ for $n=i+1,...,J$.  
For $i=1$ 
we have   $v^{(1)}_0 = \phi^{(1)}$ and 
\begin{align*}                                          
dv^{(1)}_t+\cI^{(1)}du_t
=&
\{\cL_tv^{(1)}_t
+f^{(1)}_t+\cL^{(1)}_tu_t\}\,dt \\                                                                    
&+\{\cM^{\rho}_t v^{(1)}_t+g^{(1)\rho}_t
+\cM^{(1)\rho}_t u_t\}\,dW^{\rho}_t, 
\end{align*} 
i.e.,
$$                                         
dv^{(1)}_t
=
(\cL_t v^{(1)}_t+\bar f^{(1)}_t)\,dt                                                                   
+(\cM^{\rho}_t v^{(1)}_t+\bar g^{(1)\rho}_t)\,dW^{\rho}_t, 
$$
with
\begin{align*}
\bar f_t^{(1)}:=&f_t^{(1)}-\cI^{(1)}  f_t  +(\cL^{(1)}_t-\cI^{(1)}\cL_t)  u_t,  \\
\bar g^{(1)\rho}_t:=&g^{(1)\rho}_t-\cI^{(1)}g^{\rho}_t+(\cM^{(1)\rho}_t-\cI^{(1)}\cM^{\rho}_t)  u_t. 
\end{align*}

By virtue of Theorem \ref{theorem SPDE}  this equation has a unique solution $v^{(1)}$ and 
\begin{align*}
 E\sup_{t\leq T} &  |v^{(1)}_t|_{m-1}^2  +E\int_0^T  |v^{(1)}_t|_{m}^2  \,dt \\
&\leq NE|\phi^{(1)} |^2_{m-1}+NE\int_0^T\big( |\bar f_t^{(1)}|^2_{m-2}+|\bar g^{(1)}_t|_{m-1}^2\big) \,dt. 
\end{align*}
Clearly, Lemma \ref{lemma hdiff}  implies 
\begin{align*}
& |\phi^{(1)}|^2_{m-1}\leq N |\phi|^2_{m}, \quad 
|f^{(1)}_t  |_{m-2}+|\cI^{(1)}f_t | _{m-2}\leq N |f_t |_{m-1}, \quad
 |g^{(1)\rho}_t -\cI^{(1)}g^{\rho}_t|_{m-1}\leq N|g^{\rho}_t |_{m}, \\
&
|(\cL^{(1)}_t -\cI^{(1)}\cL_t)u|_{m-2}\leq N|u|_{m+1},  
\quad 
\sum_{\rho}|(\cM^{(1)\rho}_t-\cI^{(1)}\cM^{\rho}_t)u|^2_{ m-1} \leq N^2|u|^2_{m+1}, 
\end{align*} 
with a constant $N=N(d,K,\Lambda,\psi,m)$  which does not depend on $h$. 
Hence for $m\geq 1$ 
\begin{align*}
& E\sup_{t\leq T}|v^{(1)}_t|^2_{m-1}+E\int_0^T|v^{(1)}_t|^2_{m}\,dt\\
&\qquad \leq NE|\phi|^2_{m}
+NE\int_0^T \big( |f_t|^2_{m-1}+|g_t|_{m}^2+|u_t|^2_{m+1}\big) \,dt
\leq N E  \mathfrak{K}^2_{m}. 
\end{align*}
Let $j\geq2$. Assume that for every $i<j$ the equation 
for $v^{(i)}$ has a unique solution 
such that \eqref{SPDE system} holds and that its equation can be written 
as  $v^{(i)}_0=\phi^{(i)}$ and 
$$
dv^{(i)}_t=(\cL_t  v^{(i)}_t+\bar f^{(i)}_t)\,dt+(\cM^{\rho}_t  v^{(i)}_t
+\bar g^{(i)\rho}_t)\,dW^{\rho}_t
$$
with adapted processes  $\bar f^{(i)}$ and $\bar g^{(i) \rho}$  
taking values in $H^{m-i-1}$ and 
$H^{m-i}$ respectively, such that 
\begin{equation}			\label{induc_1}
E\int_0^T \big( |\bar f^{(i)}_t|^2_{m-i-1}+|\bar g^{(i)}_t|^2_{m- i }\big) \,dt
\leq NE  \mathfrak{K}^2_m 
\end{equation}  
with a constant $N=N(K,J,m,d,T,\kappa,\Lambda,\psi)$  
independent of $h$.   Hence 
\begin{equation} 			                                           \label{induc_2}
E\Big( \sup_{t\in [0,T]} |v^{(i)}_t|_{m-i}^2 
+ \int_0^T |v^{(i)}_t|_{m+1-i}^2 dt\Big) 
\leq N E  \mathfrak{K}^2_m ,\quad i=1, ..., j-1.
\end{equation}

 Then for $v^{(j)}$ 
we have 
\begin{equation}                                                   \label{induction equation}
dv^{(j)}_t=(\cL_t v^{(j)}_t+\bar f^{(j)}_t)\,dt+
(\cM^{\rho}_t v^{(j)}_t+\bar g^{(j)\rho}_t)\,dW^{\rho}_t, 
\quad v^{(j)}_0=\phi^{(j)},
\end{equation}
with 
\begin{align*}
\bar f^{(j)}_t:=&f^{(j)}_t
+\sum_{1\leq i\leq j}  {j\choose i}  \big(\cL^{(i)}_t-\cI^{(i)}\cL_t\big)v_t^{(j-i)}-
\sum_{1\leq i\leq j}  {j\choose i}   \cI^{(i)}\bar f_t^{(j-i)}, \\
\bar g^{(j)\rho}_t:=&g^{(j)\rho}_t
+\sum_{1\leq i\leq j}  {j\choose i}   \big(  \cM^{(i)\rho}_t-\cI^{(i)}\cM_t^{\rho} \big) v_t^{(j-i)}-
\sum_{1\leq i\leq j}  {j\choose i}  \cI^{(i)}\bar g_t^{(j-i)\rho}. 
\end{align*}
Note that $|f^{(j)}_t|_{m-1-j}\leq N|f_t|_{m-1}$ ; by virtue of Lemma 
\ref{lemma hdiff},  and  by  the inequalities \eqref{induc_1} and \eqref{induc_2}   
we have 
\begin{align*}
E\int_0^T|(\cL^{(i)}_t-\cI^{(i)}\cL_t) v^{(j-i)}_t|^2_{m-j-1}\,dt
&\leq NE\int_0^T|v^{(j-i)}_t|^2_{m-j+1+i}\, dt
\leq NE  \mathfrak{K}^2_m  , 
\\
E\int_0^T|\cI^{(i)}\bar f^{(j-i)}_t|^2_{m-j-1}\,dt& 
\leq NE\int_0^T|\bar f^{(j-i)}_t|_{m-j+i-1}\,dt 
\leq NE \mathfrak{K}^2_m, 
\end{align*}
where $N=N(K,J,d,T,\kappa,\psi,\Lambda)$ 
denotes a constant which can be different 
on each occurrence. Consequently, 
$$
E\int_0^T|\bar f_t^{(j)}|_{m-j-1}^2\,dt\leq NE \mathfrak{K}^2_m, 
$$
and we can get similarly 
$$
E\int_0^T|\bar g_t^{(j)}|_{m-j}^2\,dt\leq NE \mathfrak{K}^2_m.  
$$
Hence \eqref{induction equation} has a unique solution $v^{(j)}$, and 
 Theorem \ref{theorem SPDE} implies that  the estimate \eqref{systemsol} holds 
for $v^{(j)}$ in place of $v^{(n)}$. This completes the induction and the proof 
of the theorem. 
\end{proof}
Recall that the norm $|\cdot|_{0,h}$  has been defined in \eqref{U_0h}. 
\begin{corollary}                                                                              \label{corollary2}
{\it Let Assumptions \ref{assumption coeff} and \ref{assumption data} 
hold with $m>\frac{d}{2}+J+1$ for an integer $J\geq1$.  
Let Assumptions 
\ref{assumption parab} through \ref{compatibility} be also satisfied. 
Then almost surely $v^{(i)}$ is continuous in $(t,x)\in[0,T]\times\bR^d$  for $i\leq J$, 
and its restriction 
to $\bG_h$  is an adapted continuous $\ell_2(\bG_h)$-valued process. 
Moreover, almost surely \eqref{SPDE system}-\eqref{system ini} 
hold for all $x\in\bR^d$ and $t\in[0,T]$, 
and 
$$
E\sup_{t\in[0,T]}\sup_{x}|v^{( j)}_t(x)|^2+E\sup_{t\leq T}  |v^{(j)}_t|^2_{0,h} 
\leq NE \mathfrak{K}^2_m,\quad j=1,2,..., J. 
$$
for some constant $N=N(m,J,d,T,\Lambda,\psi,\kappa)$ independent of $h$. 
}
\end{corollary}
One can obtain this corollary from Theorem \ref{theorem system} by 
a standard application of Sobolev's embedding of $H^m$ into $C^{2}_b$ for $m>2+d/2$ 
and by using the following known result; see, for example \cite{GK}, Lemma 4.2.  
\begin{lemma}		\label{lemma4.6}
Let $\varphi\in H^m$ for $m>d/2$. Then there is a constant $N=N(d,\Lambda)$ such that 
$$
| I \varphi|^2_{0,h} \leq N|\varphi|^2_m, 
$$
where $I$ denotes the Sobolev embedding operator from $H^m$ into $C_b(\bR^d)$. 
\end{lemma}

\section{Proof of Theorem \ref{theorem main2}}

Define a random field $r^h$ by 
\begin{equation}                                                                      \label{rh}
r^h_t(x):=u^h_t(x)-\sum_{0\leq i\leq J}  v_t^{(i)}(x) \frac{h^i}{i!}, 
\end{equation}
where $(v^{(1)},...,v^{(J)})$ is the solution of \eqref{SPDE system} 
and \eqref{system ini}.

\begin{theorem}    			\label{theorem expansion1}
 Let Assumptions \ref{assumption coeff} and \ref{assumption data} 
hold with $m>\frac{d}{2}+2J+2$ for an integer $J\geq0$.  
Let Assumptions 
\ref{assumption parab} through \ref{compatibility} be also satisfied.                                                         
Then $r^h$ is an $\ell_2(\bG_h)$-solution of the equation 
\begin{align}                                                
\cI^h dr^h_t(x)=&\big( \cL^h_t r^h_t(x)+F^h_t(x)\big)\,dt+\big(\cM^{h,\rho}_t r^h_t(x)
+G^{h,\rho}_t(x)\big) \,dW^{\rho}_t    ,                                               \label{vh}\\
 r^h_0(x)=&\hat\phi^{(J)h}(x),                                                                      \label{vhini}
\end{align}
where $F^h$ and $G^h$ are adapted  $\ell_2({\mathbb G}_h)$-valued such that 
  for all $h\neq 0$ with $|h|\leq 1$  
\begin{equation}                                                                      \label{FhGh}                                                       
E\int_0^T \big( |F^h_t|^2_{\ell_2(\bG_h)}
+|G^h_t|^2_{\ell_2(\bG_h)}\big)\,dt\leq N  |h|^{2(J+1)} E{\mathfrak K}^2_m, 
\end{equation}
  where $N=N(m, K,J,d,T,\kappa,\Lambda,\psi)$  is a constant 
  which does not depend on $h$. 
\end{theorem}
\begin{proof}
Using \eqref{rh}, the identity $u^h_t(x)= U^h_t(x)$ for $x\in {\mathbb G}_h$ 
which is deduced from Assumption 
\ref{assumption uU} and equation \eqref{eqU}, 
we deduce that for $x\in {\mathbb G}_h$, 
\begin{align}				                                            \label{expansion-Bis}
d \big( {\mathcal I}^h & r^h_t(x)\big) 
= d {\mathcal I}^h U^h_t - \sum_{i=0}^J \frac{h^i}{i!}\, 
{\mathcal I}^h d v_t^{(i)}(x) \nonumber  \\
=& \big[ {\mathcal L}^h_t U^h_t(x) + f^h_t(x) ] dt 
+ \big[ {\mathcal M}^{h,\rho}_t U^h_t(x) + g^{h,\rho}_t(x) \big] dW^\rho_t 
- \sum_{i=0}^J \frac{h^i}{i!}\, {\mathcal I}^h d v_t^{(i)}(x) \Big). \nonumber \\
=& {\mathcal L}^h_t  r^h_t(x) dt + \Big[ {\mathcal L}^h_t 
\sum_{i=0}^J \frac{h^i}{i!}\, v^{(i)}_t(x)  +  f^h_t(x)\Big] dt 
+ {\mathcal M}^{h,\rho}_t r^h_t(x) dW^\rho_t \nonumber  \\
&\qquad +  \Big[ {\mathcal M}^{h,\rho}_t 
\sum_{i=0}^J \frac{h^i}{i!}\, v^{(i)}_t(x)  +  g^{h,\rho}_t(x)\Big] dW^\rho_t
- \sum_{i=0}^J \frac{h^i}{i!}\, {\mathcal I}^h  d v_t^{(i)}(x).
\end{align} 
Taking into account Corollary \ref{corollary2}, 
in the definition of $d v^{(i)}_t(x) $ 
in \eqref{SPDE system} we set  
\begin{equation}		                                                                    \label{coeff_vi}
  dv^{(i)}_t(x) 
= \big[ B(i)_t(x) + F(i)_t(x) \big]dt 
+ \big[ \sigma(i)^\rho_t (x) + G(i)^\rho_t (x)\big] dW^{\rho}_t,
\end{equation}
where $B(i)_t$ (resp. $\sigma(i)^\rho_t$) 
contains the operators ${\mathcal L}^{(j)}$ 
(resp. ${\mathcal M}^{(j)\rho}_t$) for $0\leq j\leq i$
while $F(i)_t$ (resp. $G(i)^\rho_t$) contains all the free terms $f^{(j)}_t$ 
(resp. $g^{(j)\rho}_t$) for $1\leq j\leq i$.
We at first focus on the deterministic integrals.   
Using the recursive definition of the processes $v^{(i)}$
in \eqref{SPDE system}, it is easy to see that
\begin{align}		\label{recursive_coeff}
B(i)_t + \sum_{1\leq j\leq i} {i\choose j} {\mathcal I}^{(j)} 
B(i-j)_t
= &\sum_{j=0}^i {i\choose j} {\mathcal L}^{(j)}_t v^{(i-j)}_t, \\
F(i)_t + \sum_{1\leq j\leq i} {i\choose j} {\mathcal I}^{(j)} F(i-j)_t = & f^{(i)}_t.			\label{recursive_free}
\end{align}
In the sequel, to ease notations we will not mention the space variable $x$. 
 Using the expansion of ${\mathcal L}_t^h$, ${\mathcal I}^h$ 
 and the definitions of $\hat{L}_t^{(J),h}$ and $\hat{I}^{(J),h}$ 
 in \eqref{expand_LMIh}, 
 the expansion of $f^h_t$ and the definition of $\hat{f}^{(J)h}_t$ 
 given in \eqref{Taylor_free} 
 together with the definition of $d v^{(i)}_t $ in \eqref{coeff_vi}, we deduce 
\[ 
\Big[ {\mathcal L}^h_t \sum_{i=0}^J \frac{h^i}{i!}\, v^{(i)}_t  +  f^h_t\Big]  dt 
-  \sum_{i=0}^J \frac{h^i}{i!}\, {\mathcal I}^h   \big[ B(i)_t^h
+ F^{(i)}_t\big] 
= \sum_{j=1}^6 {\mathcal T}^h_t(i) dt,
\]
where 
\begin{align*}
{\mathcal T}^h_t(1)
=& \sum_{i=0}^J \sum_{j=0}^i \frac{h^j}{j!} \frac{h^{i-j}}{(i-j)!}  
\big[  {\mathcal L}^{(j)}_t v^{(i-j)}_t  -  {\mathcal I}^{(j)} 
B(i)_t  \big], \\
{\mathcal T}^h_t (2)
=& 
\sum_{i=0}^J \sum_{\stackrel{0\leq j\leq J}{i+j\geq J+1}} 
\frac{h^i}{i!} \frac{h^j}{j!}  \big[ {\mathcal L}^{(i)}_t v^{(j)}_t
- {\mathcal I}^{(i)} B(j)_t\big], \\
{\mathcal T}^h_t (3)
=& \hat{L}^{(J),h}_t \sum_{i=0}^J \frac{h^i}{i!} v^{(i)}_t  
- \hat{I}^{(J),h} \sum_{i=0}^J \frac{h^i}{i!}  B(i)_t, \\
{\mathcal T}^h_t(4)
=& \sum_{i=0}^J   \frac{h^i}{i!} f^{(i)}_t - \sum_{i=0}^J  
\sum_{j=0}^i \frac{h^j}{j!} \frac{h^{i-j}}{(i-j)!} 
{\mathcal I}^{(j)} F(i-j)_t, \\
{\mathcal T}^h_t (5)
=& -  \sum_{i=0}^J \sum_{\stackrel{0\leq j\leq J}{i+j\geq J+1}} 
\frac{h^i}{i!} \frac{h^j}{j!} {\mathcal I}^{(j)} F(i)_t, \\
{\mathcal T}^h_t(6)
=& \hat{f}^{(J)h}_t - \sum_{i=0}^J  \frac{h^i}{i!}  \hat{I}^{(J)h} f^{(i)}_t.
\end{align*}
Equation  \eqref{SPDE system}   implies 
\[ {\mathcal T}^h_t(1) = \sum_{i=0}^J \frac{h^i}{i!} 
\Big[ {\mathcal L}^{(0)}_t v^{(i)}_t  + \sum_{j=1}^{i} {i\choose j} {\mathcal L}^{(j)}_t v^{(i-j)}_t
- B(i)_t - \sum_{j=1}^{i} {i\choose j} {\mathcal I}^{(j)} B(i-j)_t\Big] .
\]
Using the recursive equation  \eqref{recursive_coeff}  
we deduce  that for every $h>0$ and $t\in [0,T]$, 
\begin{equation} 		\label{T1}
{\mathcal T}^h_t(1)=0. 
\end{equation}
A similar computation based on \eqref{recursive_free} implies 
\begin{equation} 		\label{T4}
{\mathcal T}^h_t(4)=0. 
\end{equation}

In ${\mathcal T}^h_t(2)$ all terms  have a common factor $h^{J+1}$. 
We prove an upper estimate  of 
$$
E\int_0^T | {\mathcal L}^{(i)}_t v^{(j)}_t |^2_{0,h}\,dt
$$ for $0\leq i,j\leq J$. 
Let $I$ denote the Sobolev embedding operator from 
$H^k$ to $C_b(\RR^d)$ for $k>d/2$. Lemma \ref{lemma4.6},  
inequalities \eqref{estimate_hdiff} and 
\eqref{systemsol} imply that for $k>d/2$, 
\[ 
E\int_0^T\!\!  | I {\mathcal L}^{(i)}_t v^{(j)}_t |_{0,h}^2  dt  
\leq N E\int_0^T \!\!|  {\mathcal L}^{(i)}_t v^{(j)}_t |_k^2 dt \leq
N E\int_0^T\!  |v^{(j)}_t|_{i+k+2}^2  dt \leq N E \mathfrak{K}_{i+j+k+1}^2,
\] 
where the constant $N$  does not depend on $h$ 
and  changes from one upper estimate to the next. 
Clearly, for $0\leq i,j\leq J$ with $i+j\geq J+1$, we have 
$i+j+k+1>2J+1+\frac{d}{2}$. 
Similar computations prove that for $i,j\in \{0, ..., J\}$ 
with $i+j\geq J+1$ and $k>\frac{d}{2}$, 
\begin{align*}
E\int_0^T  \big|  I {\mathcal I}^{(i)} B(j)_t 
\big|_{0,h}^2 \,  dt 
\leq & \; N \sum_{l=0}^j  E\int_0^T \big| {\mathcal L}^{(l)}_t v^{(j-l)}_t \big|_{k+i}^2 \,  dt\\
\leq & \; N \sum_{l=0}^j  E\int_0^T \big|v^{(j-l)}_t \big|_{k+i+l+2}^2  \, dt\\
\leq &\;  N E{\mathfrak K}_{k+i+j+1}^2. 
\end{align*}
These upper estimates imply the existence of some constant 
$N$ independent of $h$ such that for $|h|\in (0,1]$  and $k>\frac{d}{2}$
\begin{equation}		\label{T2}
E\int_0^T |{\mathcal T}^h_t(2)|_{0,h}^2\,ds 
\leq N |h|^{2(J+1)}  E{\mathfrak  K}_{k+2J+1}^2.
\end{equation}
We next find an upper estimate of the $|\cdot|_{0,h}$ 
norm of both terms in ${\mathcal T}_t^h(3)$. 
Using Lemmas \ref{lemma4.6}, \ref{lemma4.2} and \eqref{systemsol} 
we deduce that for $k>\frac{d}{2}$ 
\begin{align*}
E\int_0^T \Big| I \hat{L}^{(J),h}_t \sum_{i=0}^J \frac{h^i}{i!} v^{(i)}_t\Big|_{0,h}^2 dt 
\leq & 
NE\int_0^T \Big|  \hat{L}^{(J),h}_t \sum_{i=0}^J \frac{h^i}{i!} v^{(i)}_t\Big|_k^2 dt \\
\leq & N |h|^{2(J+1)}\sum_{i=0}^J \int_0^T  \big| v^{(i)}_t\big|_{k+J+3}^2 dt \\
\leq & N |h|^{2(J+1)}  E{\mathfrak K}_{k+2J+2}^2,
\end{align*}
where $N$ is a constant independent of $h$ with $|h|\in (0,1]$. 
Furthermore, similar computations yield for $k>\frac{d}{2}$ and 
$|h|\in (0,1]$ 
\begin{align*}
E\int_0^T\Big| I \hat{I}^{(J),h} \sum_{i=0}^J \frac{h^i}{i!}  B(i)_t\Big|_{0,h}^2 dt 
\leq&
N E\int_0^T \Big| \sum_{i=0}^J \frac{h^i}{i!}   \hat{I}^{(J),h} B(i)_t\Big|_{k}^2 dt \\
\leq & 
N |h|^{2(J+1)} E\int_0^T \sum_{i=0}^J 
\Big| \sum_{l=0}^i {i\choose l} {\mathcal L}^{(l)}_t v^{(i-l)}_t \Big|_{k+J+1}^2 dt \\
\leq & N  |h|^{2(J+1)} \sum_{i=0}^J \sum_{l=0}^i  |v^{(i-l)}_t|_{k+J+l+3}^2 dt \\
\leq & N |h|^{2(J+1)} E {\mathfrak K}_{k+2J+2}^2.
\end{align*}
Hence we deduce the existence of a constant $N$ independent of $h$ 
such that for $|h|\in (0,1]$, 
\begin{equation}			                                             \label{T3}
E\int_0^T |{\mathcal T}^h_t(3)|_{0,h}^2 dt  
\leq N  |h|^{2(J+1)} E {\mathfrak K}_{k+2J+2}^2,
\end{equation} 
where $k>\frac{d}{2}$. 

We next compute an upper estimate of  the $|\cdot |_{0,h}$ norm of ${\mathcal T}^h_t(5)$. 
All terms have a common factor $h^{(J+1)}$.
Recall that ${\mathcal I}^{(0)}=Id$. 
The induction equation \eqref{recursive_free} shows that 
$F(i)_t$ is a linear combination of terms of the form 
$\Phi(i)_t:=\big( {\mathcal I}^{(a_1)} \big)^{k_1} ... \big( {\mathcal I}^{(a_i)} \big)^{k_i} f_t$ 
for $a_p, k_p\in \{0, ..., i\}$   for $1\leq p\leq i$ with $\sum_{p=1}^i a_p k_p=i$, 
and of terms of the form 
$\Psi(i)_t:=\big( {\mathcal I}^{(b_1)} \big)^{l_1} ... \big( {\mathcal I}^{(b_{i-j})} \big)^{l_{i-j}} f^{(j)}_t$ for $1\leq j\leq i$, 
$b_p, l_p\in \{0, ..., i-j\}$ for $1\leq p\leq i-j$ with $\sum_{p=1}^{i-j} b_p l_p+j=i$. 
Using Lemmas \ref{lemma4.6} and \ref{lemma hdiff} we obtain for  $k>\frac{d}{2}$, $i,j=1, ... J$
\begin{align*}
E\int_0^T  | I {\mathcal I}^{(j)} \Phi(i)_t|_{0,h}^2 dt 
\leq & N E\int_0^T| {\mathcal I}^{(j)} \Phi(i)_t(x)|_k^2\, dt \\
\leq & N  E\int_0^T |\Phi(i)_t|_{k+j}^2 dt\\
\leq & N E\int_0^T |f_t|_{k+j+a_1k_1+ \cdots a_ik_i}^2 dt\\
\leq & N E\int_0^T |f_t|_{k+i+j}^2 dt \leq N E{\mathfrak K}_{k+i+j}^2.
\end{align*}
A similar computation yields
\begin{align*}
E\int_0^T  | I {\mathcal I}^{(j)} \Psi(i)_t|_{0,h}^2\, dt 
\leq & N E\int_0^T
|f^{(i)}_t|_{k+j+b_1l_1+ \cdots + b_{i-j} l_{i-j}}^2\, dt\\
\leq & NE\int_0^T |f_t|_{k+j+(i-j)+j}^2\, dt\\
 \leq& N E{\mathfrak K}_{k+i+j}^2.
\end{align*}

These upper estimates imply that for $k>\frac{d}{2}$, 
there exists some constant $N$ independent on $h$ 
such that for $|h|\in (0,1)$ 
\begin{equation}		\label{T5}
E\int_0^T  |{\mathcal T}^h_t(5)|_{0,h}^2\,dt  
\leq N  |h|^{2(J+1)}E {\mathfrak K}_{k+2J}^2.
\end{equation}

We finally prove an upper estimate  of  the $|\cdot|_{0,h}$-norm of both terms in 
${\mathcal T}_t^h(6)$.  Using Lemmas \ref{lemma4.6} and
\ref{lemma4.3}, we obtain for $k>\frac{d}{2}$,
\begin{align*}
E\int_0^T \big| I \hat{f}^{(J)h}_t\big|_{0,h}^2\, dt  
 \leq & N E\int_0^T \big| \hat{f}^{(J)h}_t\big|_k^2 dt\\
\leq &  N  |h|^{2(J+1)} E\int_0^T |f_t|_{k+J+1}^2 dt\\
\leq & N |h|^{2(J+1)} E{\mathfrak K}_{k+J+1}^2,
\end{align*}
where $N$ is a constant which does not depend on $h$.  
Furthermore, Lemmas \ref{lemma4.6} and \ref{lemma4.2} 
yield for $k>\frac{d}{2}$ and $|h|\in (0,1]$,  
\begin{align*}
E\int_0^T \Big| I \sum_{i=0}^J \frac{h^i}{i!} \hat{I}^{(J)h} f^{(i)}_t  \Big|_{0,h}^2\, dt 
\leq & N E\int_0^T \Big| \sum_{i=0}^J  \frac{h^i}{i!} \hat{I}^{(J)h} f^{(i)}_t \Big|_k^2\,dt \\
\leq & N |h|^{2(J+1)}E\int_0^T \sum_{i=0}^J |f^{(i)}_t |_{k+J+1}^2 dt \\
\leq&  N  |h|^{2(J+1)} E{\mathfrak K}_{k+2J+1}^2,
\end{align*}
for some constant $N$ independent of $h$.
Hence  we deduce that for some constant $N$ which does not depend on $h$ and $k>\frac{d}{2}$, we have for $|h|\in (0,1]$ 
\begin{equation}		\label{T6}
E\int_0^T  |{\mathcal T}^h_t(6)|_{0,h}^2\,dt 
\leq N |h|^{2(J+1)} E {\mathfrak K}_{k+2J+1}^2.
\end{equation}

Similar computations can be made for the coefficients of the stochastic integrals. 
The upper bounds of the corresponding upper
estimates in \eqref{T2} and \eqref{T3} are still valid because the operators 
${\mathcal M}^\rho_t$ are first order operators while
the operator ${\mathcal L}_t$ is a second order one. 
This implies that all operators
${\mathcal M}^{h,\rho}_t$, ${\mathcal M}^{(i)\rho}_t$ and $\hat{M}^{(J)h}_t$ 
contain less derivatives than the corresponding ones deduced from
${\mathcal L}_t$.

Using the expansion \eqref{expansion-Bis}, 
the upper estimates \eqref{T1}-\eqref{T6} for the coefficients 
of the deterministic and stochastic 
integrals, we conclude the proof.
\end{proof}
	
We now complete the proof of our main result.
\begin{proof} [Proof of Theorem  \ref{theorem main2}] 
By virtue of Theorem \ref{theorem equ} and  Theorem \ref{theorem expansion1} 
we have for $|h|\in (0,1]$ 
$$
E\sup_{t\in[0,T]}|r^h_t|_{0,h}^2
\leq 
NE|\hat\phi^{(J)h}|^2_{0,h}
+NE\int_0^T \big( |F^h|^2_{0,h}+|G_h|^2_{0,h}\big) \,dt
\leq |h|^{2(J+1)} NE{\mathfrak K}^2_m. 
$$
Using Remark \ref{rkU-h} we have $U^{-h}_t=U^h_t$  
for $t\in [0,T]$ a.s. Hence from the expansion \eqref{expansion}
  we obtain that $v^{(j)}=-v^{(j)}$ for odd $j$,  
which completes the proof of Theorem \ref{theorem main2}. 
\end{proof}

\section{Some examples of finite elements}\label{FEdim1}

In this section we propose three examples of finite elements which satisfy  Assumptions \ref{assumption invertibility},  \ref{compatibility} and \ref{assumption uU}. 

\subsection{Linear finite elements in dimension 1}  \label{fel1}
Consider the following classical linear finite elements on $\RR$ defined as follows:
\begin{equation} \label{linFE}
\psi(x)=\big(1-|x|\big)\, 1_{\{ |x|\leq 1\}}.
\end{equation} 
Let $\Lambda = \{-1, 0, 1\}$; clearly $\psi$ and $\Lambda$ satisfy
the symmetry condition \eqref{sym_psi_Lambda}. 
 Recall that  $\Gamma$ denotes the set of elements $\lambda \in {\mathbb G}$ such that the intersection of the support
of $\psi_\lambda:=\psi^1_\lambda$ and of the support of $\psi$ has a positive Lebesgue measure.
Then  $\Gamma=\{ -1,0,1\}$,  the function $\psi$ is clearly non negative, $\int_{\RR} \psi(x) dx =1$, 
$\psi(x)=0$ for $x\in \ZZ\setminus \{0\}$ and Assumption \ref{assumption uU} clearly holds.

Simple computations show that
\[ R_0=2\int_0^1 x^2 dx = \frac{2}{3}, \quad R_{-1}=R_1=\int_0^1 x(1-x) dx = \frac{1}{6}.\]
Hence  $\sum_{\lambda\in \Gamma}R_\lambda=1$. 
Furthermore, given any $z=(z_n)\in \ell_2(\ZZ)$ we have using the Cauchy-Schwarz inequality:
\[ \sum_{n\in \ZZ} \Big( \frac{2}{3} z_n^2 + \frac{1}{6} z_n z_{n-1} + \frac{1}{6} z_n z_{n+1} \Big) \geq  \frac{2}{3} \|z\|^2 - \frac{1}{6} \sum_{n\in \ZZ}
\big( z_n^2 + z_{n+1}^2\big) = \frac{1}{3}\|z\|^2.
\] 
Hence Assumption \ref{assumption invertibility} is satisfied. 
Easy computations show that for $\epsilon \in \{-1,1\}$ we have 
\[ R^{11}_0 = -2,\quad R^{11}_{\epsilon}=1,\quad R^1_0=0\quad \mbox{\rm and} \; R^1_{\epsilon}=\frac{\epsilon}{2}.\]
Hence $\sum_{\lambda\in \Gamma}R^{11}_\lambda=0$, which completes the proof of \eqref{2.5.1}. Furthermore, 
$\sum_{\lambda \in \Gamma} \lambda R^1_\lambda =1$, which proves \eqref{2.5.2} while 
$\sum_{\lambda \in \Gamma} \lambda^2 R^{11}_\lambda = 2$, which proves \eqref{2.5.3.1}. 

Finally, we have for $\epsilon \in \{-1,1\}$
\[ 
 Q^{11,11}_0= -\frac{2}{3},\quad Q^{11,11}_\epsilon= \frac{1}{3},\quad 
 \tilde{Q}^{11}_0=0\quad \mbox{\rm and} \; \tilde{Q}^{11}_\epsilon = -\frac{\epsilon}{6}.
 \] 
 This clearly implies $\sum_{\lambda \in \Gamma} Q^{11,11}_\lambda=0$ and $\sum_{\lambda\in \Gamma} \tilde{Q}^{11}_\lambda =0$, which 
 completes the proof of \eqref{2.5.4}; therefore, Assumption \ref{compatibility} is satisfied. 

\bigskip 

The following example is an extension of the previous one to any dimension.

\subsection{A general example}				\label{general_example}
Consider the following finite elements on $\RR^d$ defined as follows:
let  $\psi$ be defined on $\RR^d$ by $\psi(x)=0$ if $x\notin (-1,+1]^d$ and 
\begin{equation}                                              \label{generic}
\psi(x)=\prod_{k=1}^d \big(1-|x_k|\big) \;  \mbox{\rm for } \; x=(x_1, ..., x_d) \in (-1,+1]^d. 
\end{equation}
The function $\psi$ is clearly non negative 
 and $ \int_{\RR^d} \psi(x) dx 
 =1$.
 Let $\Lambda = \{ 0, \, \epsilon_k e_k, \, \epsilon_k\in \{-1,+1\}, \, k=1, ...,d\}$.
 Then $\psi$ and $\Lambda$ satisfy the symmetry condition \eqref{sym_psi_Lambda}.} 
Furthermore, $\psi(x)=0$ for $x\in \ZZ^d\setminus \{0\}$; 
Assumption \ref{assumption uU} clearly holds. 

%%%%%%%%%

These finite elements also satisfy all requirements in Assumptions \ref{assumption invertibility}--\ref{compatibility}.
Even if these finite elements are quite classical in numerical analysis, we were not able to find a proof of these statements in the literature.
To make the paper self-contained the corresponding easy but tedious computations are provided in an Appendix.

\subsection {Linear finite elements on triangles in the plane} \label{linea-dim2}
We suppose that $d=2$ and want to check that the following finite elements satisfy Assumptions 
\ref{assumption invertibility}-\ref{assumption uU}. For $i=1, ..., 6$, let $\tau_i$ be the
triangles defined as follows:
\begin{align}\label{def_tau_i_lin2}
& \tau_1= \{ x\in \RR^2 : 0\leq x_2\leq x_1\leq 1\}, \; \tau_2=\{ x\in \RR^2 :  0\leq x_1\leq x_2\leq 1\}, \nonumber \\
& \tau_3=\{ x\in \RR^2 : 0\leq x_2\leq 1 , x_2-1\leq x_1\leq 0\},\; 
\tau_4=\{ x\in \RR^2  : -1\leq x_1\leq x_2\leq 0\}, \nonumber \\
& \tau_5=\{x\in \RR^2 : -1\leq x_2\leq x_1\leq 0\},\;  \tau_6=\{ x\in \RR^2 : 0\leq x_1\leq 1, x_1-1\leq x_2\leq 0\}.
\end{align}
{\includegraphics[scale=1]{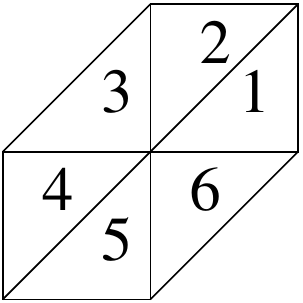}}

Let $\psi$ be the function defined by:
\begin{align} \label{def_Psi_lin2}
 &\psi(x)=1-|x_1| \, \mbox{\rm on } \tau_1\cup \tau_4, 
 \; \psi(x)=1-|x_2| \, \mbox{\rm on }  \tau_2\cup \tau_5, \nonumber  \\
 & \psi(x)=1-|x_1-x_2| \, \mbox{\rm on }   \tau_3\cup \tau_6, \;  \mbox{\rm and }\, \psi(x)=0 \; \mbox{\rm otherwise}  .
\end{align} 
 It is easy to see that the function $\psi$ is non negative 
and  that $\int_{\RR^2} \psi(x) dx=1$. Set $\Lambda = \{ 0, e_1, -e_1, e_2, -e_2\}$; 
the function $\psi$ and the set $\Lambda$
fulfill the symmetry condition \eqref{sym_psi_Lambda}.\\
  Furthermore,  
 $\Gamma = \big\{ \epsilon_1 e_1 + \epsilon_2 e_2 : (\epsilon_1, \epsilon_2)\in \{-1,0,1\}^2 , \; \epsilon_1\epsilon_2\in \{0,1\} \big\} $. 
Hence $\psi$ satisfies Assumption \ref{assumption uU}.
\smallskip

For ${\bf i}=(i_1,i_2)\in \ZZ^2$, let 
$\psi_{\bf i}$ the function defined by
 \[ \psi_{\bf i}(x_1,x_2) = \psi \big( (x_1,x_2) - {\bf i} \big) .\] 
For $\gamma =1, 2, ..., 6$, we denote by  $\tau_\gamma({\bf i})
=\big\{ (x_1,x_2) :  (x_1,x_2) - {\bf i} \in \tau_\gamma \big\}$.  Then 
\begin{align*}
 &D_1 \psi_{\bf i}= -1 \mbox{ \rm on } \tau_1({\bf i}) \cup \tau_6({\bf i})  
\mbox{ \rm and  } D_1 \psi_{\bf i}=1 \mbox{ \rm on } \tau_3({\bf i})\cup \tau_4({\bf i}), \\
&D_2 \psi_{\bf i}= -1 \mbox{ \rm on } \tau_2({\bf i}) \cup \tau_3({\bf i}) \mbox{ \rm and  } D_2 \psi_{\bf i}=1
 \mbox{ \rm on } \tau_5({\bf i})\cup \tau_6({\bf i}),\\
&D_1 \psi_{\bf i}=0 \mbox{ \rm on } \tau_2({\bf i})\cup \tau_5({\bf i}) \; 
\mbox{\rm and } D_2 \psi_{\bf i}=0 \mbox{ \rm on } \tau_1({\bf i})\cup \tau_4({\bf i}).
\end{align*}
Easy computations show that 
for ${\bf i}\in \ZZ^2$, and ${\bf k} \in \{ {\bf i}+\lambda : \lambda \in \Gamma \}$ 
\[ (\psi_{\bf i}, \psi_{\bf i})=\frac{1}{2}, \quad (\psi_{\bf i} , \psi_{\bf k}) = \frac{1}{12},\]
and $(\psi_{\bf i} , \psi_{\bf j}) =0$ otherwise. 
Thus 
\[ \sum_{\lambda\in \Gamma} R_\lambda =\sum_{\lambda\in \Gamma} (\psi, \psi_\lambda) = \frac{1}{2}+6\times \frac{1}{12}=1,\]
which proves the first identity in \eqref{2.5.1}.
 First we check that given any $\alpha \in (0,1)$ by 
Cauchy-Schwarz inequality we have some positive constants $C_1$ and $C_2$ 
such that 
\begin{align*}
|\sum_{\bf i} U_{\bf i} \psi_{\bf i} |_{L^2}^2 
\geq &   \sum_{\bf i} \int_0^\alpha dx_1 
\int_0^{x_1}  \big| (1-x_1) U_{\bf i}  + (x_1-x_2)  U_{{\bf i}+e_1} 
+ x_2 U_{{\bf i}+e_1+e_2}  \big|^2 dx_2\\
&+  \sum_{\bf i} \int_0^\alpha dx_2 \int_0^{x_2}  \big|  (1-x_2) U_{\bf i}  + (x_2-x_1) U_{{\bf i}+e_2}  
+ x_1 U_{{\bf i}+e_1+e_2}  \Big|^2 dx_1\\
\geq &  \|U\|^2 \big( \alpha^2 - C_1 \alpha^3 -C_2 \alpha^4\big) 
\end{align*}
for all $(U_{\bf i})\in\ell_2(\ZZ^2)$.   Hence, by taking $\alpha\in (0,1)$ 
such that 
$1-C_1 \alpha -C_2 \alpha^2 \geq \frac{1}{2}$, we see that 
Assumption \ref{assumption invertibility} is satisfied. 

\smallskip

We next check the compatibility conditions in Assumption \ref{compatibility}.  
Easy computations prove that for $k=1,2$ and $l\in \{1,2\}$ with $l\neq k$, 
$\epsilon_k, \epsilon_l \in \{-1,1\}$ we have
\begin{align*}
(D_k\psi, D_k\psi)=&2, \quad (D_k\psi, D_k\psi_{\epsilon_k e_k})=-1, \quad (D_k\psi, D_k\psi_{\epsilon_le_l})=0, 
\\
(D_k\psi, D_k\psi_\lambda)=&0 \; \mbox{\rm for } \lambda=\epsilon_1 e_1+\epsilon_2e_2, \; \epsilon_1\epsilon_2=1,
\end{align*}
while
\begin{align*}
(D_k \psi, D_l\psi)&=-1,\quad (D_k\psi,D_l\psi_{\epsilon_ke_k})=(D_k\psi, D_l\psi_{\epsilon_le_l})=\frac{1}{2},\\
(D_k\psi, D_l\psi_\lambda)&=-\frac{1}{2} \; \mbox{\rm for } \lambda=\epsilon_1 e_1+\epsilon_2e_2, \; \epsilon_1\epsilon_2=1.
\end{align*}
Hence for any $k,l=1,2$ and $l\neq k$ we have
\[ \sum_{\lambda\in \Gamma}(D_k\psi, D_k\psi_\lambda)=2  + 2 \times(-1)=0,\quad \sum_{\lambda\in \Gamma} 
(D_k\psi, D_l\psi_\lambda)=-1+4\times \frac{1}{2}  + 2\times \big( - \frac{1}{2}\big)  =0.\]
This completes the proof of equation $\sum_{\lambda\in \Gamma} R^{ij}_\lambda =0$ and hence of equation \eqref{2.5.1}. Furthermore,
given $k,l=1,2$ with $k\neq l$ we have for $\alpha =k$ or $\alpha =l$:
\begin{align*}
&\sum_{\lambda\in \Gamma} R^{kk}_\lambda \lambda_k\lambda_k
= - \sum_{\lambda\in \Gamma} (D_k\psi, D_k\psi_\lambda) \lambda_k\lambda_k=2\times 1^2 
=2,\\
&\sum_{\lambda\in \Gamma} R^{kk}_\lambda \lambda_l\lambda_l 
= - \sum_{\lambda\in \Gamma} (D_k\psi, D_k\psi_\lambda) \lambda_l\lambda_l=0, \\
&\sum_{\lambda\in \Gamma} R^{kk}_\lambda \lambda_k\lambda_l =
- \sum_{\lambda\in \Gamma} (D_k\psi, D_k\psi_\lambda) \lambda_k\lambda_l=0,\\
&\sum_{\lambda\in \Gamma} R^{kl}_\lambda \lambda_k\lambda_l 
= - \sum_{\lambda\in \Gamma} (D_k\psi, D_l\psi_\lambda) \lambda_k\lambda_l= \frac{1}{2} \times 1^2 +\frac{1}{2} (-1)^2 =1, \\
&\sum_{\lambda\in \Gamma} R^{kl}_\lambda \lambda_\alpha\lambda_\alpha =- \sum_{\lambda\in \Gamma} (D_k\psi, D_l\psi_\lambda) \lambda_\alpha\lambda_\alpha=0.
\end{align*}
The last identities come from the fact that $(D_k\psi, D_l\psi_{\epsilon e_k})$ ,  $(D_k\psi, D_l\psi_{\epsilon e_l})$ or
$(D_k\psi, D_l\psi_{\epsilon(e_1+e_2)}$ agree for $\epsilon=-1$ and $\epsilon=1$. This completes the proof of equation \eqref{2.5.3.1}.

We check the third compatibility condition.  Using  Lemma \ref{lemR} we deduce for $k,l=1,2$ with $k\neq l$ and $\epsilon\in \{-1,+1\}$ 
\begin{align*}
&(D_k \psi, \psi)= 0,\quad  (D_k \psi_{\epsilon e_k}, \psi)=\frac{\epsilon}{3},  \\
& (D_k \psi_{\epsilon e_l}, \psi)=- \frac{\epsilon}{6},\quad  (D_k \psi_{\epsilon(e_1+e_2)}, \psi)=  \frac{\epsilon}{6}.
\end{align*}
Therefore, using Lemma \ref{lemR} we deduce that 
\begin{align*}
&\sum_{\lambda\in \Gamma} (D_k \psi_\lambda, \psi) \lambda_k =   \frac{1}{3} + (-1)\times \big( -\frac{1}{3}\big) + \frac{1}{6} 
+  (-1)\times \big( - \frac{1}{6}\big)  =1, \\ 
&\sum_{\lambda\in \Gamma} (D_k \psi_\lambda, \psi) \lambda_l =- \frac{1}{6}  + \frac{1}{6}\times (-1) + \frac{1}{6} - \frac{1}{6} \times (-1)=0.
\end{align*}
This completes the proof of equation \eqref{2.5.2}. 

Let us check the first identity in \eqref{2.5.4}. 
Recall that 
\[ Q^{ij,kl}_\lambda=- \int_{\RR^2} z_k z_l D_i\psi(z) D_j\psi_\lambda(z) dz\]
and suppose at first that $i=j$. Then we have for  $k\neq i$, $\alpha\neq i$, $k\neq l$ and $\epsilon\in \{-1,+1\}$
\begin{align*}
Q^{ii,ii}_0=-\frac{2}{3}, &\quad Q^{ii,ii}_{\epsilon e_i}=\frac{1}{3},\quad Q^{ii,ii}_{\epsilon e_\alpha}=Q^{ii,ii}_{\epsilon (e_i+e_\alpha)}=0,\\
Q^{ii,kk}_0=-\frac{1}{3}, &\quad Q^{ii,kk}_{\epsilon e_i}=\frac{1}{6}, \quad Q^{ii,kk}_{\epsilon e_k}=Q^{ii,ii}_{\epsilon (e_i+e_k)}=0,\\
Q^{ii,kl}_0=-\frac{1}{6},& \quad Q^{ii,kl}_{\epsilon e_i}=\frac{1}{12}, \quad Q^{ii,kl}_{\epsilon e_\alpha}=Q^{ii,ii}_{\epsilon (e_i+e_\alpha)}=0.
\end{align*}
Suppose that $i\neq j$; then for $k\neq l$ and $\epsilon \in \{-1,+1\}$ we have
\begin{align*}
&Q^{ij,ii}_0=\frac{1}{6},\quad Q^{ij,ii}_{\epsilon e_j}=-\frac{1}{12},\quad Q_{\epsilon e_i}^{ij,ii}=-\frac{1}{4},\quad 
Q^{ij,ii}_{\epsilon(e_i+e_j)}=\frac{1}{4},\\
&Q^{ij,jj}_0=\frac{1}{6},\quad Q^{ij,jj}_{\epsilon e_i}=-\frac{1}{12},\quad Q_{\epsilon e_j}^{ij,jj}=-\frac{1}{12},\quad 
Q^{ij,jj}_{\epsilon(e_i+e_j)}=\frac{1}{12},\\
&Q^{ij,kl}_0=-\frac{1}{12},\quad Q^{ij,kl}_{\epsilon e_j}=\frac{1}{24},\quad Q_{\epsilon e_i}^{ij,kl}=-\frac{1}{8},\quad 
Q^{ij,  kl} _{\epsilon(e_i+e_j)}=\frac{1}{8}.\\
\end{align*}
The above equalities prove  $\sum_{\lambda\in \Gamma} Q^{ij,kl}_\lambda=0$
for any choice of $i,j,k,l=1,2$.  Hence   the first identity  in \eqref{2.5.4} is satisfied.

We finally  check the second identity in \eqref{2.5.4}.  Recall that $\tilde{Q}_\lambda^{i,k}=\int_{\RR^2} z_k D_i\psi_\lambda(z) \psi(z) dz$.
For $i=k\in \{1,2\}$, $j\in \{1,2\}$ with $i\neq j$ and $\epsilon \in \{-1,+1\}$   we have 
\[
\tilde{Q}_0^{i,i}= - \frac{3}{12}, \quad 
\tilde{Q}_{\epsilon e_i}^{i,i}=  \frac{3}{24}, \quad 
\tilde{Q}_{\epsilon e_j}^{i,i}= -  \frac{1}{24},\quad 
\tilde{Q}_{\epsilon (e_i+e_j)}^{i,i}=  \frac{1}{24}.
\] 
Hence $\sum_{\lambda\in \Gamma} \tilde{Q}_\lambda^{i,i}=0$. Let $i\neq k$; then for $\epsilon\in \{-1,+1\}$ we have 
\[ \tilde{Q}_0^{i,k} = \tilde{Q}_{\epsilon e_i}^{i,k}=0, \quad  \tilde{Q}^{i,k}_{\epsilon e_k}=- \frac{1}{12},\quad 
\tilde{Q}^{i,k}_{\epsilon (e_i+e_k)}= \frac{1}{12}. \]
Hence $\sum_{\lambda\in \Gamma} \tilde{Q}_\lambda^{i,k}=0$ for any choice of $i,k=1,2$, which concludes the proof of  \eqref{2.5.4}
Therefore, the function $\psi$ defined by \eqref{def_Psi_lin2} satisfies all Assumptions \ref{assumption invertibility}-\ref{assumption uU}.

\section{Appendix}			\label{appendix}
 The aim of this section is to prove that the example described in \ref{general_example} 
 satisfies Assumptions \ref{assumption invertibility} and \ref{compatibility}.

%%%%%%%%%

 For $k=1, ..., d$, let $e_k\in \ZZ^d$ denote the $k$-th unit vector of $\RR^d$; then $\bG=\ZZ^d$ and
\[\Gamma=\Big\{ \sum_{k=1}^d \epsilon_k e_k \; : \; \epsilon_k\in \{-1,0,1\} \, \mbox{\rm for } k=1, ... , d \Big\}. \]
For fixed $k=1, ... , d$ (resp. $k\neq l \in \{1, ... , d\}$) let
\begin{equation}				\label{I(k)}
 {\mathcal I}(k)=\{1, ... , d\} \setminus \{k\}, \quad \mbox{\rm resp. } {\mathcal I}(k,l)=\{1, ... , d\} \setminus \{k,l\}.
 \end{equation}

Note that in the particular case $d=1$, the functions $\psi$ gives rise to the classical linear finite elements.
Then for ${\bf i} \in {\mathbb Z}^d$, we have for $k=0, 1, ... , d$:
\begin{equation} \label{psi_i,psi_j}  
R_{\bf i}:= (\psi_{\bf i},\psi)= 
{d\choose k} \Big(\frac{1}{6}\Big)^k 
\Big( \frac{2}{3}\Big)^{d-k} \quad \mbox{\rm if } \; \sum_{l=1}^d |i_l|=k.
\end{equation}
Furthermore, given 
$k=0, 1, ... , d$, there are $2^k$ elements ${\bf i}\in \ZZ^d$ such that 
$\sum_{l=1}^d |i_l|=k$. 
Therefore, we deduce 
\[ 
\sum_{{\bf i}\in \ZZ^d} (\psi_{\bf i},\psi)= \sum_{k=0}^d 2^k {d\choose k} \Big(\frac{1}{6}\Big)^k \Big( \frac{2}{3}\Big)^{d-k}
= {d\choose k} \Big(\frac{2}{6}\Big)^k \Big( \frac{2}{3}\Big)^{d-k}=1,
\] 
which yields the first compatibility conditon in \eqref{2.5.1}.

We at first check that Assumption \ref{assumption invertibility} holds true, that is
\[ \delta \sum_{{\bf i}\in \ZZ^d} U_{\bf i}^2 = \delta |U|^2_{\ell_2(\ZZ^d)} 
\leq \Big|\sum_{{\bf i}\in \ZZ^d} U_{\bf i} \psi_{\bf i}\Big|_{L^2}^2 
=\sum_{{\bf i}, {\bf j}\in \ZZ^d} R_{{\bf i}-{\bf j}} U_{\bf i} U_{\bf j},  
\quad   U\in \ell_2(\ZZ^d).\]
for some $\delta >0$. 
 For $U\in \ell_2({\mathbb Z}^d)$ and $k=1, ... , d,$ let $T_k U=U_{e_k}$,  
 where $e_k$ denotes the $k$-th vector of the canonical basis.

For $U\in \ell_2(\ZZ^d)$ we have 
\begin{align*}
 \Big| \sum_{\bf i} &U_{\bf i} \psi_{\bf i}\Big|_{L^2}^2 
 = \sum_{\bf i\in \ZZ^d} \int_{[0,1]^d} 
 \Big[ U_{\bf i} \prod_{j=1}^d  ( 1-{x_j}) + \sum_{k=1}^d (T_k U)_{\bf i}\;  x_k 
 \prod_{j\in {\mathcal I}(k)}(1-x_j) \\
 & + \sum_{1\leq k_1<k_2\leq d} (T_{k_1} \circ T_{k_2} U)_{\bf i}\;  x_{k_1} x_{k_2} 
 \prod_{j\in {\mathcal I}(k_1,k_2)}(1-x_j) +\cdots  + (T_1 \circ T_2 \cdots \circ T_d U)_{\bf i} \prod_{k=1}^d x_k\Big]^{2} dx, 
\end{align*}
Given $\alpha \in (0,1)$ if we let 
\[ I(\alpha)\!=\!\!\int_0^\alpha \! (1-x)^2 dx 
= \alpha - \alpha^2 + \frac{\alpha^3}{3}, 
\; J(\alpha)\!=\!\! \int_0^\alpha\!  x(1-x) dx = 
\frac{\alpha^2}{2}-\frac{\alpha^3}{3}, 
\; K(\alpha)\!=\!\! \int_0^\alpha\!  x^2 dx = \frac{\alpha^3}{3} ,\]
restricting the above integral on the set $[0,\alpha]^d$, 
expanding the square and using the Cauchy Schwarz inequality 
we deduce the existence of some constants $C(\gamma_1,\gamma_2, \gamma_3)$ 
defined for $\gamma_i\in \{0,1, ... , d\}$ such that 
\begin{align*}
  \Big| \sum_{\bf i} U_{\bf i} \psi_{\bf i}\Big|_{L^2}^2  \geq &\sum_{\bf i} |U_{\bf i}|^2 
  \Big[ I(\alpha)^d + {d\choose 1} K(\alpha)\, I(\alpha)^{d-1} + {d\choose 2} K(\alpha)^2\,
 I(\alpha)^{d-2} + \cdots + K(\alpha)^d\Big] \\
 & -2 \Big( \sum_{\bf i}  |U_{\bf i}|^2 \Big) 
 \sum_{\gamma_1+\gamma_2+\gamma_3=d, \gamma_2+\gamma_3\geq 1} C(\gamma_1,\gamma_2, \gamma_3)\; I(\alpha)^{\gamma_1}\;
 J(\alpha)^{\gamma_2} \; K(\alpha)^{\gamma_3}\\
 \geq & |U|_{\ell_2({\mathbb Z}^d)}^2   
 \Big( \alpha^d -  \sum_{l=d+1}^{3d} C_l \alpha^l \Big), 
 \end{align*}
 where $C_l$ are some positive constants. Choosing $\alpha$ small enough, we have
 $ \big| \sum_{\bf i} U_{\bf i} \psi_{\bf i}\big|_{L^2}^2  \geq \frac{\alpha^d}{2}
|U|_{\ell_2({\mathbb Z}^d)}^2$, 
 which implies the invertibility Assumption \ref{assumption invertibility}.

We now prove that the compatibility Assumption \ref{compatibility} holds true. 
For $l=1, ... , d$, $n=0, ... , d-1$:
\begin{align}\label{D_lpsi_icase1}
 (D_l \psi_{\bf i} , D_l\psi) &= - \, 2^{d-1-n} \, \Big( \frac{1}{6}\Big)^n \, 
 \Big( \frac{1}{3}\Big)^{d-1-n}  \; \mbox{ for } |i_l|=1,  \;
 \sum_{r \neq l, 1\leq n\leq d} |i_r|=n,
 \\ \label{D_lpsi_icase0}
 (D_l \psi_{\bf i} , D_l \psi) &= + \, 2^{d-n} \, \Big(\frac{1}{6}\Big)^n \, 
 \Big( \frac{1}{3}\Big)^{d-1-n}  \; \mbox{ for } |i_l|=0,  \; 
 \sum_{r\neq l, 1\leq r\leq d} |i_r|=n. 
  \end{align} 
For $n=1, ... , d-1$ and $k_1<k_2<...  < k_n$ with $k_r\in {\mathcal I}(l)$,  
where   ${\mathcal I}(l)$ is defined in
\eqref{I(k)},  let
\begin{align*}
 \Gamma_l(k_1, ... , k_n)&=\Big\{ \sum_{r=1}^n \epsilon_{k_r} e_{k_r} \; : \; \epsilon_{k_r}\in \{ -1, 1\}, r=1, ... , n \Big\} ,\\
\Gamma_l(l;k_1, ... , k_n)&=\Big\{ \epsilon_l e_l + \sum_{r=1}^n \epsilon_{k_r} e_{k_r} \; : \; \epsilon_l \in \{-1,1\} \; 
\mbox{\rm and } \epsilon_{k_r}\in \{ -1, 1\}, r=1, ... , n\Big\} .
\end{align*}
Then $| \Gamma_l (k_1, ... , k_n)|=2^n$ while  $| \Gamma_l (l;k_1, ... , k_n)|=2^{n+1}$.  
For $l=1, ... , d$, the identities \eqref{D_lpsi_icase1} and \eqref{D_lpsi_icase0} imply 
\begin{align*}
\sum_{\lambda\in \Gamma}& (D_l \psi, D_l \psi_\lambda) = \Big[ (D_l\psi, D_l\psi) + \sum_{\epsilon_l\in \{-1,+1\}}
 (D_l \psi, D_l \psi_{\epsilon_l e_l})\Big]\\
& \quad + \sum_{n=1}^{d-1} \sum_{k_1<k_2<...  <k_n, k_j\in {\mathcal I}(l)} \Big[ \sum_{\lambda \in \Gamma_l(k_1,... , k_n)}
(D_l\psi, D_l\psi_\lambda )  + \sum_{\lambda \in \Gamma_l(l; k_1, ... , k_n)}  (D_l\psi, D_l\psi_\lambda) \Big]  \\
& = \Big[ 2^d \Big( \frac{1}{3}\Big)^{d-1} -2\times 2^{d-1} \Big( \frac{1}{3}\Big)^{d-1}\Big] \\
&\quad +\sum_{n=1}^{d-1}  2^n {d-1 \choose n} \Big[ 2^{d-n} \Big( \frac{1}{6} \Big)^n \Big( \frac{1}{3}\Big)^{d-1-n} -2\times 2^{d-1-n} \Big(\frac{1}{6}\Big)^n  \Big( \frac{1}{3}\Big)^{d-1-n}\Big] =0 .
\end{align*}
This proves the second identity in \eqref{2.5.1} when $i=j$. Furthermore, \eqref{D_lpsi_icase1} implies 
\begin{align*}
\sum_{\lambda\in \Gamma} R^{ll}_\lambda \lambda_l \lambda_l &= - 
\sum_{\lambda\in \Gamma} (D_l \psi, D_l \psi_\lambda) \lambda_l \lambda_l = - \sum_{\epsilon_l\in \{-1,1\}}
\big(D_l \psi, D_l \psi_{\epsilon_l e_l}  \big)
 \\
 &\quad 
- \sum_{n=1}^{d-1} \sum_{k_1<k_2<...  <k_n, k_j\in {\mathcal I}(l)} \sum_{\lambda\in 
\Gamma_l(l;k_1, ... , k_n)} (D_l \psi, D_l \psi_\lambda)\\
&= 2\times 2^{d-1} \Big(\frac{1}{3}\Big)^{d-1} + \sum_{d=1}^n {d-1 \choose n} 2^{n+1} \times 2^{d-1-n} \Big( \frac{1}{6}\Big)^n
\Big( \frac{1}{3}\Big)^{d-1-n}  \\
&= 2 \sum_{n=0}^{d-1} {d-1 \choose n} \Big(\frac{2}{6}\Big)^n \Big(\frac{2}{3}\Big)^{d-1-n} =2,\quad l=1,... ,d. 
\end{align*}
 Furthermore, given $k\neq l \in \{1, ... , d\}$, 
\begin{align*} 
\sum_{\lambda\in \Gamma} R^{ll}_\lambda \lambda_k\lambda_k&=
- \sum_{\lambda \in \Gamma} (D_l \psi, D_l \psi_\lambda) \lambda_k \lambda_l=0.
\end{align*}
Indeed, for $n=1, ... , d-1$, $k_1<...  < k_n $ where $k_r\in {\mathcal I}(l)$ and at least one of the indices $k_r$ is equal to $k$ for
$r=1, ... ,n$, given $\lambda\in \Gamma_l(k_1, ... , k_n)$ we have using \eqref{D_lpsi_icase1} and \eqref{D_lpsi_icase0}
\[ \sum_{\epsilon_l\in \{-1,1\}} (D_l\psi,D_l\psi_{\epsilon_le_l+\lambda}) \lambda_k \lambda_l =
 - 2^{d-1-n} \Big(\frac{1}{6}\Big)^n \Big(\frac{1}{3}\Big)^{d-1-n} \times  (-1+1)=0.\]
This proves the second identity in \eqref{2.5.3.1} when both derivatives agree. 

Also note that for $k\neq l \in \{1, ... , d\}$ we have $\sum_{\lambda \in \Gamma} R^{kl}_\lambda=0$. 
Indeed, for  $\lambda$ as above 
\begin{align*}  (D_k \psi, D_l \psi_\lambda)  &+ \sum_{\epsilon_l\in \{-1,1\}} (D_k\psi, D_l\psi_{\epsilon_le_l+\lambda}) \\
&= 2^{d-n} \Big(\frac{1}{6}\Big)^n \Big(\frac{1}{3}\Big)^{d-1-n} -2\times 2^{d-1-n} \Big(\frac{1}{6}\Big)^n \Big(\frac{1}{3}\Big)^{d-1-n} =0,
\end{align*}
while $R^{kl}_\lambda=0$ for other choices of $\lambda \in \Gamma$.

We now study the case of mixed derivatives. 
Given $k\neq l \in \{1, ... , d\}$ recall that  ${\mathcal I}(k,l)=\{ 1, ... , d\} \setminus \{k,l\}$. 
Then for $k\neq l \in \{1, ... , d\}$ and ${\bf i} \in \ZZ^d$ we have for $n=0, ... , d-2$
\begin{align}
(D_k \psi_{\bf i}, D_l \psi )&=0 \; \mbox{ \rm if } |i_k \, i_l|\neq 1, \label{DkDl0}\\
(D_k \psi_{\bf i}, D_l \psi)&= - \Big(\frac{1}{2}\Big)^2 \; \Big( \frac{1}{6}\Big)^n \; \Big(\frac{2}{3}\Big)^{d-n-2} \; 
\mbox{\rm if }\;  i_k\, i_l=1,  \quad 
 \sum_{r\in {\mathcal I}(k,l)} |i_r|=n, \label{DkDl2} \\
(D_k \psi_{\bf i}, D_l \psi)&= + \Big(\frac{1}{2}\Big)^2 \; \Big( \frac{1}{6}\Big)^n \; \Big(\frac{2}{3}\Big)^{d-n-2} \; 
\mbox{\rm if }\;  i_k \, i_l=- 1, \quad 
\sum_{r\in {\mathcal I}(k,l)} |i_r|=n. \label{DkDl1}
\end{align}

For $n=1, ... , d-2$ and $k_1<...  <k_n$ with $k_r\in {\mathcal I}(k,l)$ for $r=1,..., n$, set
\[ \Gamma_{k,l}(k_1, ... , k_n)=\Big\{ \sum_{r=1}^n \epsilon_{k_r} e_{k_r} \; : \; \epsilon_r\in \{-1,1\} \Big\}.\]
For $n=0$ there is no such family of indices $k_1<... <k_n$ and we let
 $\Gamma_{k,l}(\emptyset)=\{0\}$. Thus for $n=0, ... , d-2$, $| \Lambda_{k,l}(k_1, ..., k_n)|=2^n$. 
Using the identities \eqref{DkDl0}-\eqref{DkDl1} we deduce
 \begin{align}\label{sum_DkDl}
\sum_{\lambda\in \Gamma} &(D_k \psi, D_l \psi_\lambda) = \sum_{n=0}^{d-2} \; \sum_{k_1<k_2<...  <k_n, k_r\in {\mathcal I}(k,l)} 
\sum_{\lambda \in \Gamma_{k,l}(k_1,... , k_n)} \big[ (D_k\psi, D_l\psi_{e_k+e_l+\lambda}) \nonumber \\
&\quad +   (D_k\psi, D_l\psi_{e_k-e_l+\lambda})+
(D_k\psi, D_l\psi_{-e_k+e_l+ \lambda}) + (D_k\psi, D_l\psi_{-e_k-e_l+\lambda}) \big]  \nonumber \\
=&  \sum_{n=0}^{d-2} {d-2 \choose n} \, 2^n \, \Big[ - \Big( \frac{1}{2}\Big)^2 \, \Big(\frac{1}{6}\Big)^n \, \Big(\frac{2}{3}\Big)^{d-2-n}
+  \Big( \frac{1}{2}\Big)^2 \, \Big(\frac{1}{6}\Big)^n \, \Big(\frac{2}{3}\Big)^{d-2-n} \nonumber \\
&\quad +
 \Big( \frac{1}{2}\Big)^2 \, \Big(\frac{1}{6}\Big)^n \, \Big(\frac{2}{3}\Big)^{d-2-n}
  -  \Big( \frac{1}{2}\Big)^2 \, \Big(\frac{1}{6}\Big)^n \, \Big(\frac{2}{3}\Big)^{d-2-n}\Big] =  0,\quad k\neq l.
\end{align}
This completes the proof of the second identity in \eqref{2.5.1} when $i\neq j$, and hence \eqref{2.5.1} holds true.
Furthermore, the identities \eqref{DkDl2} and \eqref{DkDl1} imply   for  $i\neq j \in \{1, ... , d\}$
 and $\{i,j\}=\{k,l\}$  
 \begin{align*}
\sum_{\lambda\in \Gamma} &(D_k \psi, D_l \psi_\lambda)\lambda_k \lambda_l 
= \sum_{n=0}^{d-2} \; \sum_{k_1<k_2<...  <k_n, k_r\in {\mathcal I}(k,l)} 
\sum_{\lambda \in \Gamma_{k,l}(k_1,... , k_n)} \big[ (D_k\psi, D_l\psi_{e_k+e_l+\lambda}) \\
&\quad - (D_k\psi, D_l\psi_{e_k-e_l+\lambda})
- (D_k\psi, D_l\psi_{-e_k+e_l+ \lambda}) + (D_k\psi, D_l\psi_{-e_k-e_l+\lambda}) \big] \\
=&  -  4 \, \Big( \frac{1}{2}\Big)^2 \sum_{n=0}^{d-2} {d-2 \choose n} \,  2^n  \, \Big(\frac{1}{6}\Big)^n \, \Big(\frac{2}{3}\Big)^{d-2-n} 
=-  \sum_{n=0}^{d-2} {d-2 \choose n} \, \Big(\frac{2}{6}\Big)^n \Big( \frac{2}{3}\Big)^{d-2-n}= -1.
\end{align*}
Equation \eqref{DkDl0} proves that $(D_k \psi, D_l \psi_\lambda)=0$ if $|\lambda_k \lambda_l|\neq 1$. Hence using 
\eqref{sum_DkDl} we deduce  that for any $r=1, ... , d$, 
\[  \sum_{\lambda\in \Gamma} (D_k \psi, D_l \psi_\lambda)\lambda_r \lambda_r =0.\]
Let $r\in {\mathcal I}(k,l)$ and for $n=1, ... , d-3$, let $k_1<...  < k_n$ be such that $k_j\in \{1, ... , d\} \setminus \{k,l,r\}$
and $\lambda = \sum_{j=1}^n \epsilon_{k_j} e_{k_j}$ for $\epsilon_{k_j}\in \{-1,1\}$, $j=1, ... , n$. 
Then for any choice of $\epsilon_k$ and $\epsilon_l$ in $\{-1,1\}$ the equalities \eqref{DkDl2} and \eqref{DkDl1} imply that
\[ (D_k\psi, D_l \psi_{\lambda + \epsilon_k e_k + \epsilon_l e_l + e_r} )  
=  (D_k\psi, D_l \psi_{\lambda + \epsilon_k e_k + \epsilon_l e_l - e_r} )   .\]
This clearly yields that for $r\in {\mathcal I}(k,l)$ we have
\[ \sum_{\lambda\in \Gamma} (D_k \psi, D_l \psi_\lambda)\lambda_k \lambda_r =  
\sum_{\lambda\in \Gamma} (D_k \psi, D_l \psi_\lambda)\lambda_l \lambda_r =0.\] 
Finally, given $n=2, ... , d$ and  $k_1<... <k_n$ where
the terms $k_j\in {\mathcal I}(k,l)$, then given any choice of $\epsilon_k$ and $\epsilon_l$ in $\{-1,1\}$, the value of 
$(D_k \psi, D_l \psi_{\epsilon_k e_k + \epsilon_l e_l + \lambda})$ does not depend on the value of 
$\lambda\in \Gamma_{k,l}(k_1, ..., k_n)$. 
Therefore, if we fix  $r_1\neq r_2$ in the set ${\mathcal I}(k,l)$, for fixed $n$  there are as many choices of indices 
$k_1<... <k_n$ such that $\epsilon_{r_1} \epsilon_{r_2}=1$ that of such indices such that
$\epsilon_{r_1} \epsilon_{r_2}=-1$. This proves
\[ \sum_{\lambda \in \Gamma} (D_k\psi, D_l\psi_\lambda) \lambda_{r_1} \lambda_{r_2}=0,\]
which completes the proof of the first identity in \eqref{2.5.3.1} for mixed derivatives; hence \eqref{2.5.3.1} holds true. 

We now check the compatibility condition \eqref{2.5.2}. 
Fix $j\in \{1, ... , d\}$; then
\begin{equation}\label{Dj0}
(D_j\psi, \psi)= 2^{d-1} \Big( \prod_{k\neq j} \int_0^1 (1-x_k)^2 dx_k\Big) \Big[  \int_0^1 (-1) (1-x_j) dx_j 
+  \int_{-1}^0 (1+x_j) dx_j\Big] =0, 
\end{equation}
while
\begin{align} \label{ej}
(D_j \psi, \psi_{e_j})&=2^{d-1} \Big( \prod_{k\neq j} \int_0^1 (1-x_k)^2 dx_k\Big) \int_0^1  (-1) \big(1+(x_j-1)\big) dx_j = - \frac{1}{2}
\Big(\frac{2}{3}\Big)^{d-1}, \nonumber \\
(D_j \psi, \psi_{-e_j})&=2^{d-1} \Big( \prod_{k\neq j} \int_0^1 (1-x_k)^2 dx_k\Big) \int_{-1}^0  \big(1- (x_j+1)\big) dx_j =  \frac{1}{2}
\Big(\frac{2}{3}\Big)^{d-1}. 
\end{align}
For $n=1, ... , d-1$ and $k_1<... <k_n$ where the indexes $k_r$, $r=1, ... , n$ are different from $j$ we have for any 
$\lambda\in \Gamma_j(k_1,... , k_n)$ 
\begin{align}\label{Dj-noej}
(D_j\psi, \psi_{\lambda})  =& 2^{d-(n+1)} \Big(\prod_{k\in \Gamma\setminus \{j,k_1, ... , k_n\}} \int_0^1 (1-x_k)^2 dx_k\Big)
\times \Big( \prod_{r=1}^n \int_0^1 x_{k_r} (1-x_{k_r}) dx_{k_r} \Big) \nonumber  \\
&\times \Big[  \int_0^1 (-1)  (1-x_j) dx_j + \int_{-1}^0   (1+x_j) dx_j\Big] = 0,
\end{align} 
while 
\begin{align}\label{ej+lambda}
(D_j\psi, \psi_{e_j+\lambda})  =& 2^{d-(n+1)} \Big(\prod_{k\in \Gamma\setminus \{j,k_1, ... , k_n\}} \int_0^1 (1-x_k)^2 dx_k\Big)
\times \Big( \prod_{r=1}^n \int_0^1 x_{k_r} (1-x_{k_r}) dx_{k_r} \Big) \nonumber \\
&\times \int_0^1 (-1) \big( 1+ (x_j-1)\big) dx_j = -\frac{1}{2} \;  \Big( \frac{2}{3}\Big)^{d-(n+1)} \; \Big( \frac{1}{6}\Big)^n ,
\end{align} 
and 
\begin{align}\label{-ej+lambda}
(D_j\psi, \psi_{-e_l+\lambda})  =& 2^{d-(n+1)} \Big(\prod_{k\in \Gamma\setminus \{j,k_1, ..., k_n\}} \int_0^1 (1-x_k)^2 dx_k\Big)
\times \Big( \prod_{r=1}^n \int_0^1 x_{k_r} (1-x_{k_r}) dx_{k_r} \Big) \nonumber  \\
&\times \int_{-1}^0  \big( 1- (x_j+1)\big) dx_j = \frac{1}{2} \;  \Big( \frac{2}{3}\Big)^{d-(n+1)} \; \Big( \frac{1}{6}\Big)^n . 
\end{align}
Note that the number of terms $ (D_j\psi, \psi_{\epsilon_l e_l+\lambda})   $ with $\epsilon_l=-1$ or $\epsilon_l =+1$ is equal to 
${d-1 \choose n} 2^n$.
Therefore, the identities \eqref{Dj0}-\eqref{-ej+lambda} imply that for any $j=1, ... , d$ we have
\begin{align} \label{compat_iii-j}
\sum_{\lambda\in \Gamma}& \lambda_j R^j_\lambda =- \sum_{\lambda\in \Gamma} (D_j \psi, \psi_\lambda)\; \lambda_j = 
 \frac{1}{2} \Big(\frac{2}{3}\Big)^{d-1}  -  \frac{1}{2} \Big(\frac{2}{3}\Big)^{d-1} (-1) \\
& + \frac{1}{2} \sum_{n=1}^{d-1} {d-1 \choose n}  2^n   
\Big( \frac{2}{3}\Big)^{d-1-n}
\Big( \frac{1}{6}\Big)^n  
-  \frac{1}{2}\sum_{n=0}^{d-1} {d-1 \choose n} 2^n  \Big( \frac{2}{3}\Big)^{d-1-n}
\Big( \frac{1}{6}\Big)^n \times (-1) \nonumber \\
 = &   \sum_{n=0}^{d-1} {d-1 \choose n} \Big( \frac{2}{6}\Big)^n \Big(\frac{2}{3}\Big)^{d-1-n} = 1.
\end{align}
This proves \eqref{2.5.2} when $i=k$. 

Let $k\neq j \in \{1, ..., d\}$ and given $n=1, ... , d-1$ let $k_1<... < k_n$ be indices that belong to ${\mathcal I}(j)$
such that one of the indices $k_r, r=1, ..., n$ is equal to $k$. Given any $\lambda\in \Gamma_j(k_1, ..., k_n)$ we deduce that
\[ (D_j \psi, \psi_{e_l+\lambda}) \lambda_k + (D_j \psi, \psi_{-e_l+\lambda}) \lambda_k =0.\]
This completes the proof of the  identity \eqref{2.5.2}. 

In order to complete the proof of the validity of Assumption \ref{compatibility}, it remains to check that the identities in \eqref{2.5.4} hold true. 
Recall that for $\lambda\in \Gamma$ and $i,j,k,l\in \{1, ...,d\}$ we have 
\[ Q_\lambda^{ij,kl}=\int_{\RR^d} z_k z_l D_j\psi_\lambda(z)  D_i^\ast \psi(z)  dz
=-\int_{\RR^d} z_k z_l D_j\psi_\lambda (z) D_i \psi(z) dz.\]
For $p=1, ..., 4$, $n=1,..., d-p$ and $i_1,..., i_p \in \{1, ... ,d\}$ with $i_1, ...,  i_p$ pairwise different let
\begin{align*}
 \mathcal{I}_n(i_1, ..., i_p):=\Big\{ &\sum_{\alpha=1}^n \epsilon_\alpha e_{k_\alpha} ; \epsilon_\alpha\in \{-1,+1\}, 
 \; 1\leq k_1<... <k_n\leq d, \\
& k_\alpha \not\in \{i_1, ...,   i_p\}  \, \mbox{\rm for } \alpha=1, ..., n\Big\}, 
\end{align*}
and $ \mathcal{I}_0(i_1, ..., i_p)=\{0\}$. 

First suppose that $i=j$. \\
First let $k=l=i$; then for $n=0, ..., d-1$ and $ \mu\in \mathcal {I}_n(i)$ we have 
\[ Q_{\mu }^{ii,ii} + Q_{\mu + e_i}^{ii,ii} + Q_{\mu - e_i}^{ii,ii}=0.\]
Let  $k=l$ with $k\neq i$; then then for  $n=0, ..., d-1$ and $\mu \in \mathcal{I}_n(i)$ we have
\[ Q_{\mu}^{ii,kk} + Q_{\mu + e_i}^{ii,kk} + Q_{\mu -e_i}^{ii,kk}=0. \]
Let  $l=i$ and $k\neq i$; then  for $n=0, ..., d-2$, $\epsilon \in \{-1,+1\} $ and $ \mu\in \mathcal {I}_n(i,k)$ we have
\[ Q_{\mu + \epsilon e_i + e_k}^{ii,ki} + Q_{\mu + \epsilon e_i - e_k}^{ii,ki}=0.\]
 A similar result holds for $k=i$ and $l\neq i$.  
Furthermore,   $Q_\lambda^{ii,ki}=0$ is $\lambda$ is not equal to $\mu + \epsilon e_i + e_k$ or $\mu + \epsilon e_i - e_k$ for 
$\mu \in \mathcal{I}_n(i,k)$ for some $n$.

Let  $k\neq l$ with $k\neq i$ and $l\neq i$; then for $n=0, ..., d-2$, $\epsilon \in \{-1,+1\}$ and $\mu\in \mathcal {I}_n(k,l)$
we have 
\[ Q_{\mu + \epsilon e_k + e_l}^{ii,kl} + Q_{\mu + \epsilon e_k - e_l}^{ii,kl}=0,\]
while $Q_\lambda^{ii,kl}=0$ is $\lambda$ is not equal to $\mu + \epsilon e_i + e_k$ or $\mu + \epsilon e_i - e_k$ for  $\mu \in 
\mathcal {I}_n(i,k)$ for some $n$.

We now suppose that $i\neq j$. \\
First suppose that  $k=i$ and $l=j$; then for  $n=0, ..., d-1$ and $\mu \in \mathcal{I}_n(i)$ we have
 \[ Q_{\mu}^{ij,ij} + Q_{\mu + e_j}^{ij,ij} + Q_{\mu -e_j}^{ij,ij}=0. \] 
Let $k=l=i$; then for $n=0, ..., d-2$, $\epsilon \in \{-1+1\}$
 and $\mu \in \mathcal {I}_n(i,j)$ we have
\[ Q_{\mu + \epsilon e_i +e_j}^{ij,ii} + Q_{\mu + \epsilon e_i -e_j}^{ij,ii}=0,\]
while $Q_\lambda^{ij,ii}=0$ is $\lambda$ is not equal to $\mu+\epsilon e_i +e_j$ or $\mu + \epsilon e_i-e_j$ where $\mu\in \mathcal {I}_n(i,j)$
for some $n$. A similar result holds exchanging $i$ and $j$ for $k=l=j$. \\
Let  $k=l$ with $k\not\in \{i,j\}$ and $l\not\in \{i,j\}$; then for $n=0, ..., d-2$, $\epsilon\in \{-1,+1\}$  and $\mu\in \mathcal{I}_n(i,j)$  we have 
\[ Q_{\mu + \epsilon e_i +e_j}^{ij,kk} + Q_{\mu + \epsilon e_i -e_j}^{ij,kk}=0,\]
while $Q_\lambda^{ij,kk}=0$ is $\lambda$ is not equal to $\mu+\epsilon e_i +e_j$ where $\mu + \epsilon e_i-e_j$ for $\mu\in \mathcal {I}_n(i,j)$
for some $n$. \\
Let $l=i$ and $k\not\in \{i,j\}$; then for  $n=0, ..., d-2$, $\epsilon \in \{-1+1\}$
 and $\mu \in \mathcal {I}_n(i,k)$ we have
\[ Q_{\mu + \epsilon e_i +e_k}^{ij,ki} + Q_{\mu + \epsilon e_i -e_k}^{ij,ki}=0,\]
while $Q_\lambda^{ij,ki}=0$ is $\lambda$ is not equal to $\mu+\epsilon e_i +e_k$ or $\mu + \epsilon e_i-e_k$ where
 $\mu\in \mathcal {I}_n(i,k)$
for some $n$. A similar result holds exchanging $i$ and $j$ for $k=l=j$. \\
Finally, let $k\neq l$ with $k\not\in \{i,j\}$ and $l\not\in \{i,j\}$;    then for  $n=0, ..., d-4$, $\epsilon_i, \epsilon_j, \epsilon_k \in \{-1+1\}$
 and $\mu \in \mathcal {I}_n(i,j,k,l)$ we have
\[ Q_{\mu + \epsilon_i e_i +\epsilon_j e_j + \epsilon_k e_k +e_l}^{ij,kl} + Q_{\mu + \epsilon_i e_i +\epsilon_j e_j + \epsilon_k e_k -e_l}^{ij,kl}=0,\]
while $Q_\lambda^{ij,kl}=0$ is $\lambda$ is not equal to $\mu + \epsilon_i e_i +\epsilon_j e_j + \epsilon_k e_k +e_l$ or
$\mu + \epsilon_i e_i +\epsilon_j e_j + \epsilon_k e_k -e_l$  where $\mu\in \mathcal {I}_n(i,j,k,l)$
for some $n$.
These computations complete the proof of the first identity in \eqref{2.5.4}. Recall that for $i,k\in \{1, ..., d\}$ and $\lambda\in \Gamma$ 
we let
\[ \tilde{Q}_\lambda^{i,k}:=\int_{\RR^d} z_k D_i\psi_\lambda(z) \psi(z) dz. \]
Let $k=i$; for $n=0, ..., d-1$ and $\mu\in \mathcal{I}_n(i)$ we have
\[ \tilde{Q}^{i,i}_\mu + \tilde{Q}^{i,i}_{\mu + e_i} + \tilde{Q}^{i,i}_{\mu-e_i}=0.\]
Let $k\neq i$; for $n=0, ..., d-2$,  $\epsilon\in \{-1,0, +1\}$   and $\mu\in \mathcal{I}_n(i,k)$ we have
\[ \tilde{Q}_{\mu + \epsilon e_i + e_k}^{i,k} + \tilde{Q}_{\mu + \epsilon e_i - e_k}^{i,k}=0\]
while $\tilde{Q}_\lambda^{i,k}=0$ if $\lambda$ is not equal to $\mu+\epsilon e_i +e_k$ or $\mu + \epsilon e_i-e_k$ where
 $\mu\in \mathcal {I}_n(i,k)$ for some $n$.
 This completes the proof of the second identity in \eqref{2.5.4}; therefore Assumption \ref{compatibility}
 is satisfied for these finite elements. This completes the verification of the validity of Assumptions \ref{assumption invertibility}-\ref{compatibility} 
 for the function  $\psi$ defined by \eqref{generic}.
 \bigskip

 \noindent{\bf Acknowledgements} This work started while Istv\'an Gy\"ongy was invited professor at the University Paris 1 Panth\'eon Sorbonne.
  It was completed when Annie Millet was invited by the University of Edinburgh. Both authors want to thank the University Paris 1, the 
  Edinburgh Mathematical Society and the Royal Society of Edinburgh for their financial support. 
The authors want to thank anonymous referees for their careful reading and helpful remarks.

 %%%%%%%%%%%%%%%%%%%

\end{document}